\RequirePackage{fix-cm}
\documentclass{svjour3}

\usepackage{amsmath,amssymb,amsfonts}%
\usepackage{graphicx}%
\usepackage{multirow}%
\usepackage{mathrsfs}%
\DeclareMathAlphabet{\mathscr}{U}{rsfs}{m}{n}
\usepackage{xcolor}%
\usepackage{booktabs}%
\usepackage{hyperref}
\usepackage{algorithm}%
\usepackage{algorithmic}
\usepackage{listings}%
\usepackage[backend=biber,sorting=none]{biblatex} 
\usepackage{mfirstuc}
\usepackage{subcaption}
\usepackage{caption}
\usepackage{easyReview}
\usepackage{bm}
\usepackage[perpage]{footmisc}

\newtheorem{assumption}{Assumption}

\newcommand{\Bcal}{{\cal B}}
\newcommand{\Ccal}{{\cal C}}

\newcommand{\Ical}{{\cal I}}

\newcommand{\Kcal}{{\cal K}}

\newcommand{\Mcal}{{\cal M}}

\newcommand{\Ocal}{{\cal O}}

\newcommand{\Rcal}{{\cal R}}
\newcommand{\Scal}{{\cal S}}

\newcommand{\Wcal}{{\cal W}}

 \newcommand{\sign}{\textrm{sgn}}                  
 \newcommand{\Proj}{\mathbb{P}}

\newcommand{\bx}{x}

\newcommand{\bH}{H}
\newcommand{\bI}{I}

\begin{document}
\titlerunning{Second-order iteratively reweighted $\ell_1$ method}
\title{Alternating Iteratively Reweighted $\ell_1$ 
 and Subspace Newton Algorithms for Nonconvex Sparse Optimization}







\author{Hao Wang       \and
        Xiangyu Yang \and
        Yichen Zhu
}

\institute{    
           Hao Wang \at
              School of Information Science and Technology, ShanghaiTech University, Shanghai, 201210, China \\
              \email{haw309@gmail.com}           
           \and
           Xiangyu Yang \at
                School of Mathematics and Statistics, Henan University, Kaifeng, 475000, China\\
                Center for Applied Mathematics of Henan Province, Henan University, Zhengzhou,  450046, China\\
              \email{yangxy@henu.edu.cn}
            \and
            Yichen Zhu (Corresponding author)\at
            School of Information Science and Technology, ShanghaiTech University, Shanghai, 201210, China \\
              \email{zhuych2022@shanghaitech.edu.cn}      
}

\date{Received: date / Accepted: date}
\maketitle
\begin{abstract}
This paper presents a novel hybrid algorithm for minimizing the sum of a continuously differentiable loss function and a nonsmooth, possibly nonconvex, sparse regularization function. The proposed method alternates between solving a reweighted $\ell_1$-regularized subproblem and performing an inexact subspace Newton step. The reweighted $\ell_1$-subproblem allows for efficient closed-form solutions via the soft-thresholding operator, avoiding the computational overhead of proximity operator calculations. As the algorithm approaches an optimal solution, it maintains a stable support set, ensuring that nonzero components stay uniformly bounded away from zero. It then switches to a perturbed regularized Newton method, further accelerating the convergence. We prove global convergence to a critical point and, under suitable conditions, demonstrate that the algorithm exhibits 
local linear and quadratic convergence rates. Numerical experiments show that our algorithm outperforms existing methods in both efficiency and solution quality across various model prediction problems.

\end{abstract}

\section{Introduction}
In this paper, we consider optimization problems of the form
\begin{equation}\label{p.prob} \tag{$\mathscr{P}$}
	\begin{aligned}
		& \underset{x\in \mathbb{R}^n}{\min}
		& & f(x)+ \lambda h(x),
	\end{aligned}
\end{equation}
where $f:\mathbb{R}^{n}\to \mathbb{R}$ is proper and twice continuously differentiable but possibly nonconvex, $\lambda>0$ refers to the regularization parameter, and the regularization function $h:\mathbb{R}^{n} \to [0,+\infty)$ is a sum of composite functions as follows
\begin{equation}\label{eq:regularizationTerm}
    h(x):= \sum_{i=1}^{n}(r \circ \vert \cdot \vert)(x_i), \ \forall x \in \mathbb{R}^{n}. 
\end{equation}
The function $r$  satisfies  the following assumptions:
\begin{assumption}\label{assp_f_and_r}
    $r:\mathbb{R}\to \mathbb{R}_+ $ is $\Ccal^2$-smooth on $\mathbb{R}\setminus\{0\}$ with $r(0)=0$,  $r'(t) > 0$ and $r''(t) \leq 0$ for $t \in \mathbb{R}_{++}$. 
\end{assumption}


Assumption \ref{assp_f_and_r} encompasses several common surrogates for the $\ell_0$-norm. Typical examples (see  Table \ref{tab.l0approx}) include the $\ell_{p}$ regularization  with $p \in (0,1)$ \cite{fazel2003log}, Smoothly Clipped Absolute Deviation \cite{fan2001variable}, Minimax Concave Penalty \cite{zhang2010nearly}, and Cappled $\ell_1$ \cite{zhang2010analysis}, among others. These regularizations consistently yield less biased estimates for nonzero components in applications such as sparse learning, model compression, compressive sensing, and signal/image processing (see, e.g., \cite{fan2001variable,chartrand2008restricted,liu2007sparse,chen2010smoothing}).

 \begin{table}[htbp]
	\centering
	\caption{Examples of regularized functions $(0<p<1)$. We use $r'(|x_i|)$ and $r''(|x_i|)$ to denote the first- and second-order derivative of $r$ with respect to $|x_i|$.}
	  {   \begin{tabular}{lllccc}
	  \toprule
	  Regularizer  & $h(|x|)$ &  $   r'(|x_i|)$ & $r'(0^{+})$& $r''(|x_i|)$ & $r''(0^{+})$\\
	  \midrule
   	  LPN \cite{fazel2003log}   & $\sum^n_{i=1}(|x_i|)^p$  & $p(|x_i|)^{p-1}$ & $+\infty$ &$p(p-1)(|x_i|)^{p-2}$ & $-\infty$\\
	  LOG \cite{lobo2007portfolio}  &  $\sum^n_{i=1}\log(1+\frac{|x_i|}{p}),$     &  $\frac{1}{|x_i| + p}$ & $\frac{1}{p}$& $-\frac{1}{(|x_i|+p)^2}$ &$-\frac{1}{p^2} $\\
	  FRA \cite{fazel2003log}  &   $\sum^n_{i=1}\frac{|x_i|}{|x_i|+p}$ &  $\frac{p}{(|x_i|+p)^2}$ & $\frac{1}{p}$& $-\frac{2p}{(|x_i|+p)^3}$&$-\frac{2}{p^2} $\\
	  TAN \cite{candes2008enhancing}  & $\sum^n_{i=1}\arctan(\frac{|x_i|}{p})$  &  $\frac{p}{p^2 + (|x_i|)^2}$ & $\frac{1}{p}$& $-\frac{2p|x_i|}{(p^2 +|x_i|^2)^2}$& 0\\
	  EXP \cite{bradley1998feature}   &  $\sum^n_{i=1}1-e^{-\frac{|x_i|}{p}}$     &  $\frac{1}{p}e^{-\frac{|x_i|}{p}}$ & $\frac{1}{p}$&  $-\frac{1}{p^2}e^{-\frac{|x_i|}{p}}$&$-\frac{1}{p^2} $ \\
	  \bottomrule
	  \end{tabular}}
	\label{tab.l0approx}%
  \end{table}%

While first-order methods have been developed to address problems \eqref{p.prob}, 
they often exhibit slow convergence and have not fully exploited the underlying sparse structure. 
In this work, we propose a hybrid algorithm that overcomes these limitations by combining reweighted $\ell_1$ subproblems with subspace Newton iterations.
The reweighted $\ell_1$ subproblems facilitate the identification of an optimal subspace,  on which Newton iterations can be applied to accelerate local convergence.

\subsection{Related Work} 
{\bf First-order methods.} Numerous first-order methods have been developed for nonconvex sparse optimization problems. A prominent focus has been the $\ell_p (0<p<1)$ regularization, recognized as a challenging and notable representative regularizer in \eqref{eq:regularizationTerm}. A widely-used approach for $\ell_p$-regularized problems is the forward-backward (FB) algorithm, which updates iterates by solving a proximal subproblem. However, closed-form \emph{global} solutions to the subproblem  are only available for specific values of $p$, such as $p=1/2$ and $p=2/3$ \cite{xu2012l}, while for other values of $p$ and different regularization terms, solving the subproblems requires iterative methods \cite{liu2024bisection, won2023unified}, introducing additional computational overhead.
To address the inefficiency of solving nonconvex subproblems, several first-order methods based on smoothing techniques have been proposed. For instance, smoothing approximations \cite{chen2010smoothing,lai2013improved,lu2014iterative,wang2021relating}
such as  $|x_i|^p \approx (x_i^2+\epsilon)^{p/2}$, simplify the problem but often sacrifice accuracy, so that subproblems with a sequence of $\epsilon$ approaching  zero need to be solved. 
Another approach is the  \emph{iteratively reweighted} $\ell_1$ (IR$\ell_1$) method \cite{wang2021relating}, which approximates  the $\ell_p^p$ regularization  by a sequence of weighted $\ell_1$ norms. This method benefits from closed-form subproblem solutions via soft-thresholding and is easily applied to various nonconvex regularization problems. 
In addition, it exhibits \emph{support identification} properties, enabling it to rapidly detect zero components and preserving the sign of iterates near optimal solutions  \cite{wang2021relating, wang2023convergence, wang2022extrapolated}. Our approach builds upon these reweighting strategies while incorporating second-order information to further enhance convergence rate and solution quality.

{\bf Second-order methods.}  Second-order methods, such as proximal Newton methods \cite{lee2014proximal,yue2019family,mordukhovich2023globally}, offer a more robust framework for nonconvex optimization. However, relatively few second-order methods have been specifically designed for nonconvex sparse optimization problems. An early contribution by Chen \cite{chen2013optimality} introduced the Smoothing Trust Region Newton (STRN) method, which applies smoothing techniques to penalty functions and solves trust-region subproblems. While STRN guarantees global convergence to points satisfying affine-scaled second-order necessary conditions, it lacks local convergence analysis and often requires careful tuning of the smoothing parameter.
Recent advances have focused on subspace and manifold-based Newton methods \cite{bareilles2023newton, wu2023regularized, zhou2023revisiting}. These methods leverage support detection or manifold identification techniques to restrict optimization to a reduced subspace, thereby reducing the computational cost of Newton steps. For example, the Manifold Newton Method \cite{bareilles2023newton} employs proximal gradient steps to identify the manifold containing the optimal solution, followed by Newton updates restricted to that manifold.

{\bf Hybrid methods.}
 Recently, first-order methods augmented with second-order information gained significant attention. These methods rely primarily on first-order techniques in the early stages and employ second-order methods to improve local convergence. 
One closely related method is the Hybrid Proximal Gradient and Subspace Regularized Newton Method (HpgSRN) \cite{wu2023regularized}, considered a competing approach to the Manifold Newton Method \cite{bareilles2023newton}. This approach  alternates between proximal gradient steps and subspace Newton steps,  thereby combining the efficiency of first-order methods with the accuracy of second-order updates. In addition, it is shown to achieve local superlinear convergence under the Kurdyka-\L{}ojasiewicz (KL) property and a local error bound condition.
Similarly, the Proximal Semismooth Newton Pursuit (PCSNP) method \cite{zhou2023revisiting} employs support detection with a more aggressive criterion for identifying nonzero components, which can lead to solving an excessive number of unnecessary quadratic subproblems in the early stages.

\subsection{Contributions} 
We propose a novel hybrid approach, the subspace \textbf{I}teratively \textbf{Re}weighted $\ell_1$ and subspace \textbf{N}ewton \textbf{A}lgorithm (IReNA), to solve \eqref{p.prob} by alternating between solving a reweighted $\ell_1$ subproblem and a subspace Newton subproblem. The algorithm employs stationarity measures to differentiate between updating the zero and nonzero components of the current iterate, thereby optimizing over the most relevant subspace. Specifically, it first solves the reweighted $\ell_1$ subproblem within a subspace to determine whether the support of the iterates has stabilized. Once stabilization is detected, the algorithm switches to solving an (inexact) Newton subproblem restricted to the subspace of nonzero components for improving local convergence. We demonstrate the effectiveness of this approach through extensive experiments on diverse model prediction problems.  Compared to existing second-order methods, the proposed algorithm presents several advantages. The main novelties of IReNA, when compared with STRN \cite{chen2013optimality}, HpgSRN \cite{wu2023regularized} and PCSNP \cite{zhou2023revisiting}, include: 

\begin{enumerate}

\item[(i)] {\bf  General applicability to nonconvex regularization}: Our method employs an iteratively reweighted $\ell_1$ subproblem, making it applicable to a broad class of nonconvex regularization problems---beyond just specific $\ell_p$ norms. The subproblem admits a closed-form solution via the soft-thresholding operator, eliminating the need for subproblem solvers typically required for nonconvex proximal gradient subproblems.

\item[(ii)] 
{\bf Efficient subspace optimization}: 
Our approach optimizes within an adaptively chosen subspace. By evaluating the optimality residuals for both nonzero and zero components at each iteration, the algorithm dynamically adjusts the subspace to maximize progress while minimizing computational cost, particularly in the early stages.

\item[(iii)]  
{\bf Stable sign preservation}:  The algorithm guarantees that the iterates maintain their sign, keeping nonzero components uniformly bounded away from zero near an optimal solution. Once the signs stabilize, the algorithm smoothly transitions to a perturbed Newton method that optimizes a smoothed approximation of the original problem.

\item[(iv)] 
{\bf Comprehensive convergence guarantees}: We establish global convergence to a critical point  and demonstrate local superlinear convergence under the KL property. Furthermore, local quadratic convergence is achieved when the exact Newton step is employed. The convergence to second-order optimal solutions is also established when using a trust-region Newton subproblem.

\item[(v)]  {\bf Superior performance in experiments}: Extensive numerical experiments show that our algorithm consistently outperforms existing hybrid approaches that combine proximal gradient and Newton methods, achieving higher efficiency and improved solution quality.
\end{enumerate}

We summarize the comparison of convergence results for the related first- and second-order algorithms in Table \ref{tab.related}. 

\begin{table}[htbp]
  \centering
  \caption{Comparison of convergence results for state-of-the-art algorithms for solving \eqref{p.prob}.  (a) $f$ is Lipschitz differentiable on a bounded set. (b) $f$ is strongly smooth. (c) $f$ is Lipschitz twice differentiable on the support set. (d) Curve ratio condition. (e) KL property. (f) Positive definiteness of generalized Hessian. (g) $\nabla f$ is strongly semismooth on a bounded set. (h) $f$ is smooth and convex. (i) Nonsingular local minimizer.}
    \begin{tabular}{clll}
    \toprule
    Problems. & Algorithms  & Global convergence & Local convergence  \\
    \midrule
    \multirow{2}[2]{*}{$h(x) = \|x\|_p^p$, $p \in \{\frac{1}{2}, \frac{2}{3}\}$} & HpgSRN\cite{wu2023regularized} & (c)(d)(e) & superlinear, (e) \\
          & PCSNP\cite{zhou2023revisiting} & (a)   & quadratic, (f)(g) \\
    \midrule
    \multirow{2}[2]{*}{$h(x) = \|x\|_p^p$, $p \in (0,1)$} & IR$\ell_\alpha$\cite{lu2014iterative} & (b)  & - \\
          & EPIR$\ell_1$\cite{wang2022extrapolated} & (a),(e)  & linear/superlinear, (e) \\
    \midrule
    \multirow{4}[2]{*}{Generic $h(x)$} & AIR \cite{wang2021nonconvex} & (h)  & linear, (h) \\
          & STRN \cite{chen2013optimality} & (c)     &  - \\
          & \multirow{2}[1]{*}{IReNA (Ours)} & \multirow{2}[1]{*}{(a)} & superlinear, (e) \\
          &       &       & quadratic, (c)(i) \\
    \bottomrule
    \end{tabular}%
  \label{tab.related}%
\end{table}%

%
\subsection{Notation} \label{sec.preliminaries}
We denote by $\mathbb{R}^{n}$ the $n$-dimensional Euclidean space, and by $\mathbb{R}_{+}^n$ and $\mathbb{R}_{++}^n$ the sets of nonnegative and strictly positive real vectors, respectively.
For any $x, y\in \mathbb{R}^n$, we use $x_i$ to denote the $i$th component of $x$ and define $[n]:= \{1,2,3,\ldots,n\}$. 
Let $x\circ y$ denote the Hadamard product of $x$ and $y$ and   $x \leq y$  denote $x_i \leq y_i, \ \forall i\in[n]$. The index sets of nonzero and zero components of $x$ are defined as 
$\Ical(x) = \{ i \mid x_i \neq 0\}$ and $\Ical_0(x) = \{i\mid x_i = 0\}$, respectively.
For any subset $\Scal \subseteq [n]$, the subvector of $x$ containing the entries indexed by $\Scal$ is denoted by $x_{\Scal} \in \mathbb{R}^{|\Scal|}$.  The $\ell_p$ norm of a vector $x$ is given by $\|x\|_p := \left( \sum_{i=1}^n |x_i|^p \right)^{1/p}$, and  $\|\cdot\|$ represents the $\ell_2$ norm.

Consider a function $F:\mathbb{R}^{n} \to \mathbb{R}$. For any $x_{\Scal} \neq 0$, the partial gradient of $F$ over $x_{\Scal} $ is denoted as $\nabla_{\Scal} F(x)$, i.e. , ${\nabla_{\Scal} F(x) = \frac{\partial F}{\partial x_{\Scal}}} \in \mathbb{R}^{\vert \Scal \vert}$.  
We avoid denoting  $\nabla_i F(x) = [\nabla F(x)]_i$ for $i \in [n]$ since $F$ may not be differentiable at $x$. 
For any vector $x \in \mathbb{R}^{n}$,  the signum function $\sign: \mathbb{R}^n\to\mathbb{R}^n$    yields a vector whose components are the signum of the individual components of $x$,  $\textrm{sgn}(x_i)=1$ if $x_i > 0$,  $\textrm{sgn}(x_i)=-1$ if $x_i < 0$ and $\textrm{sgn}(x_i)=0$ if $x_i = 0$.  The weighted soft-thresholding operator is defined as $[\mathbb{S}_{\omega}(v)]_i:= \textrm{sgn}(v_i)\max(|v_i| - \omega_i,0)$ for any $v\in\mathbb{R}^n$ and $\omega \in \mathbb{R}^n_{++}$. 
The projection operator \( \mathbb{P}(y;x) \) projects \( y \) onto the subspace orthant containing \( x \), defined as:
$[\mathbb{P}(y;x)]_i := \sign(x_i) \max\{0, \sign(x_i) y_i\}.$
Define the ball $\mathcal{B}(x,\rho):=\{y \in \mathbb{R}^n\mid  \|y-x\|\leq \rho\}$ with $\rho>0$. If function $f: \mathbb{R}^n\rightarrow \mathbb{R} \cup \{+\infty\}$ is convex, then the subdifferential of $f$ at $\bar{x}$ is given by
$\partial f(\bar{x}) := \{z \mid f(\bar{x}) + \langle z, x - \bar{x}\rangle \leq  f(x),\ \forall x \in \mathbb{R}^n \}.$
For a set $S$, the relative interior $\textrm{rint}(S)$ is defined as
$\textrm{rint}(S) := \{x \in S : \text{there exists } \rho > 0\text{ such that } \Bcal(x,\rho) \cap \textrm{aff}(S) \subseteq S\}$
where aff($S$) is the affine hull of $S$.

\section{Proposed IReNA}\label{sec.algorithm}

We first recall the definition of the first-order stationary points of \eqref{p.prob}.
\begin{definition}[Stationary point {\cite{wang2021nonconvex}}]\label{Def_StatioanaryPoint}
    We say that a vector $x^{*} \in \mathbb{R}^{n}$ is a first-order stationary point of \eqref{p.prob} if 
    \begin{equation}\label{eq.optimalcondition}
      0 \in \nabla_i f(x^{*})+\lambda \partial r(|x_i^{*}|),\ \forall i\in [n].
    \end{equation}
\end{definition}

To formulate the subproblem in our proposed algorithm, we define a smooth and locally Lipschitz continuous approximation:
\begin{equation}\label{eq.relaxedlp}
	\begin{aligned}
		& \underset{x\in \mathbb{R}^n}{\textrm{min}}
		& & F(x; \epsilon) := f(x)+ \lambda\sum_{i=1}^n r(|x_i|+\epsilon_i),
	\end{aligned}
\end{equation}
where $\epsilon \in \mathbb{R}^n_{++}$. Clearly $F(x) = F(x;0)$. 
A point $x \in \mathbb{R}^{n}$ is a first-order stationary point of \eqref{eq.relaxedlp} if 
\begin{equation}\label{eq.relaxedopcondition2}
 \begin{aligned}
    &|\nabla_i f( x) | \le  \omega_i,  \ &i\in \Ical_0( x), \\
    &\nabla_i f( x) +  \omega_i \textrm{sgn}( x_i) = 0, &i\in\Ical(x),
    \end{aligned}
\end{equation}
where {   $ \omega_i =  \omega(x_i,\epsilon_i) = \lambda r'(| x_i|+\epsilon_i),\ \forall i \in [n]$}.  
The relationship between \eqref{eq.optimalcondition} and \eqref{eq.relaxedopcondition2} is as follows: \eqref{eq.relaxedopcondition2} holds at $x$ with $\epsilon_i = 0, \forall i \in \Ical(x)$ if and only if \eqref{eq.optimalcondition} holds at $x$. 
At the $k$th iteration, we define a locally weighted $\ell_1$ regularized function as
\begin{equation}\label{eq.Gk}
    \begin{aligned} G_{k}(x) = f(x)+ \sum_{i=1}^n \omega^k_i|x_i|,
    \end{aligned}
\end{equation}
where $\omega^k_i=  \omega(x_i^k,\epsilon_i^k) 
, \forall i \in [n]$. 
We say that $x \in \mathbb{R}^{n}$ is a first-order stationary point of problem \eqref{eq.Gk} if
\begin{equation}\label{eq.Gkopcondition}
    0 \in \nabla_i f(x) + \omega^k_i \partial |x_i|, \ \forall i \in [n].
\end{equation}
The following lemma establishes the relationship   between the optimality of $G_{k}(x)$ and those of $F(x;\epsilon^k)$.  

\begin{lemma}\label{lem.FG}  
Consider \eqref{eq.relaxedlp} and \eqref{eq.Gk}. For any $(x^k, \epsilon^k), (x^{k+1}, \epsilon^{k+1}) \in \mathbb{R}^n \times \mathbb{R}_{++}^n$ , it holds for any $\epsilon^{k+1} \le \epsilon^k$ that 
    \begin{equation} \label{lem.dec0}
        F(x^{k+1}; \epsilon^{k+1}) - F(x^k; \epsilon^{k}) \le F(x^{k+1};\epsilon^k) - F(x^k; \epsilon^{k}) \leq G_k(x^{k+1})-G_k(x^k).
    \end{equation}
 Moreover, the following statements are equivalent:
        \begin{itemize}
            \item[(i)] $x^k$ is first-order stationary for $G_k(\bx)$.
            \item[(ii)] $x^k$ is  first-order stationary for $F(\bx; \epsilon^{k})$.
            \item[(iii)] $x^k =\mathbb{S}_{\omega^k }(x^k - \nabla f(x^k))$.
        \end{itemize}
\end{lemma}
\begin{proof} 
  The first inequality of \eqref{lem.dec0} is obvious. The second inequality is true since 
 $$   r(|x_i^{k+1} |  +\epsilon_i^{k})- r(|x_i^{k} |  +\epsilon_i^{k}) \le \omega^k_i(|x_i^{k+1}| - |x^k_i|)$$ 
  by the concavity of $r(|\cdot|)$ over $\mathbb{R}_{++}$.  Moreover, the equivalence holds true by noting that
\begin{equation*}
\begin{aligned}
    (i) \overset{\eqref{eq.Gkopcondition}}{\iff} &\ -\nabla f(x^k) \in (\omega^{k}\circ \partial\Vert x^k \Vert_1) 
    \iff \   \left(x^k -\nabla f(x^k)\right)\in (x^k +  \omega^{k} \circ \partial \Vert x^k\Vert_1) \\
     \iff & \  (iii) 
     \iff \  \eqref{eq.relaxedopcondition2} \textrm{ is satisfied}
     \iff \ \textrm{(ii). } 
 \end{aligned} 
\end{equation*}
\end{proof}


\subsection{Optimality Measure and Subspace Determination}

Our approach solves different subproblems based on the optimality residuals associated with the current zero and nonzeros. 
For this purpose,    we further partition the support of $x$ into 
$  \Ical_+(x):= \{ i \mid x_i > 0\} $ and $ \Ical_-(x):= \{i \mid x_i < 0\}$. %
Motivated by \cite{chen2017reduced}, we introduce two optimality residuals, $\Psi(x; \epsilon)$ and $\Phi(x;\epsilon)$, corresponding to the residuals for the
 zeros and nonzeros at $(x, \epsilon)$, respectively. These definitions also takes into account how far nonzero components might shift before they turn to zero. 

\begin{equation}\label{eq:zerosmeasure}
    [\Psi(x; \epsilon)]_i := 
    \begin{cases}
    \nabla_if(x)+\omega_i,& \textrm{if } i \in \Ical_0(x) \textrm{ and 
 }  \nabla_i f(x)+\omega_i < 0,\\
    \nabla_if(x)-\omega_i,& \textrm{if } i \in \Ical_0(x)  \textrm{ and } \nabla_i f(x) - \omega_i > 0,\\
    0, & \textrm{otherwise,}
    \end{cases}
\end{equation}
\begin{equation}\label{eq:nonzerosmeasure}
   \begin{aligned} & 
    [ \Phi(x; \epsilon)]_i := \\ 
    & \begin{cases}
    0, &\textrm{if } i \in \Ical_0(x),\\
    \min\{\nabla_i f(x)+\omega_i, \max \{x_i, \nabla_i f(x)-\omega_i\}\}, & \textrm{if } i \in \Ical_+(x) \text{ and } \nabla_i f(x)+\omega_i > 0, \\
    \max\{\nabla_i f(x)-\omega_i, \min \{x_i, \nabla_i f(x)+\omega_i \}\}, & \textrm{if } i \in \Ical_-(x) \text{ and } \nabla_i f(x)-\omega_i < 0, \\
    \nabla_i f(x) + \omega_i\cdot \textrm{sgn}(x_i), & \textrm{otherwise.}
    \end{cases}
    \end{aligned} 
\end{equation}

We abbreviate $\Phi(x^k;\epsilon^k)$ and $\Psi(x^k;\epsilon^k)$ as $\Phi^k$ and $\Psi^k$ , $\Ical(x^k)$ and $\Ical_0(x^k)$ as $\Ical^k$ and $\Ical_0^k$ respectively. 
The following lemma summarizes their properties.

\begin{proposition}\label{prop:optimility}
Consider \eqref{eq:zerosmeasure} and \eqref{eq:nonzerosmeasure}. For any $(x, \epsilon) \in \mathbb{R}^n \times \mathbb{R}^n_{++}$ and  $\omega = \omega(x, \epsilon)$, the following statements hold.
\begin{itemize}
    \item[(i)] $ |[\Phi(x; \epsilon)]_i |  \le \left| \nabla_i  F(x; \epsilon) \right|,\ i\in\Ical(x). $
   \item[(ii)] Let $d(\beta) :=  \mathbb{S}_{\beta   \omega }(x -  \beta  \nabla f(x)) - x$ for $\beta > 0$, then      
    $$d(1) = -\left[\Psi(x; \epsilon) + \Phi(x; \epsilon)\right]. $$ 
    More generally,  
\begin{subequations}
    \begin{align}
        d_i(\beta)  &= - \beta  [\Psi(x ; \epsilon )]_i, \qquad\qquad\    i \in \Ical_0(x ),  \label{le.dk1} \\
         |d_i(\beta)|   & \ge  \min\{\beta ,1\} | [\Phi(x ; \epsilon )]_i |, \quad  i \in \Ical(x ). \label{le.dk2}
    \end{align}
\end{subequations}
    \item[(iii)] 
	  $x$ satisfies the first-order stationary condition  \eqref{eq.relaxedopcondition2} if $ \|\Phi(x; \epsilon)\|=\|\Psi(x; \epsilon)\|=0. $
	 Additionally, if $\epsilon_i = 0$, $i\in\Ical (x) $, then this implies that $x$ satisfies the first-order stationary condition \eqref{eq.optimalcondition}. 
 \end{itemize}
\end{proposition}
\begin{proof} 
(i) We consider the last three cases in \eqref{eq:nonzerosmeasure}. For $i \in \Ical_+(x)$, note that by \eqref{eq.relaxedopcondition2} we have $\nabla_i F(x;\epsilon) = \nabla_i f(x) + \omega_i$. When $i \in \Ical_+(x) \textrm{ and }\nabla_i f(x)+\omega_i > 0$, we also have that if $\nabla_i f(x) -\omega_i<0$, then $[\Phi(x;\epsilon)]_i =\min\{ \nabla_i f(x) + \omega_i, x_i\} \le \nabla_i f(x) + \omega_i$. Otherwise $[\Phi(x; \epsilon)]_i =\min\{ \nabla_i f(x) + \omega_i,  \nabla_i f(x) - \omega_i\} \le \nabla_i f(x) + \omega_i$. As for the last case, it is obvious that $[\Phi(x; \epsilon)]_i = \nabla_i f(x) + \omega_i =  \nabla_i F(x; \epsilon)$.  On the other hand, similar arguments can be applied when $i \in \Ical_-(x) \textrm{ and }\nabla_i f(x)-\omega_i < 0$. Overall,  we have shown that (i) holds. 

(ii) The proof of the first statement can be found in \cite[Lemma A.1]{chen2017reduced}. It indicates that 
$[\Psi(x ; \epsilon )]_i = - d_i(1)$, $i\in\Ical_0(x)$ and $[\Phi(x ; \epsilon )]_i = - d_i(1)$, $i\in\Ical(x)$ by the complementary of $\Psi$ and $\Phi$. 
This, together with Lemma \ref{lem.ist.ist} shown in Appendix, implies \eqref{le.dk1} and \eqref{le.dk2}.  

(iii) The statement holds true by noting that
\begin{equation*}
    \|\Phi(x; \epsilon)\|=\|\Psi(x; \epsilon)\|=0 \iff d(1)=0 \iff  \textrm{ Lemma \ref{lem.FG}(iii) holds }.
\end{equation*} Moreover, if $\epsilon_i = 0, i\in \Ical(x)$, then $x$ satisfies the first-order stationary condition \eqref{eq.optimalcondition}. 

\end{proof}

One reason for employing different optimality residuals for zero and nonzero components is that our approach is a subspace method of minimizing   \emph{subsets} of zero and nonzero components at each iteration. 
These two residuals are used as the ``switching sign'' for choosing which subsets to update. 
Specifically, at each iteration, the subset $\Wcal_k$ is chosen to satisfy
\begin{subequations}\label{eq.chooseW}
    \begin{align}
        \Wcal_k \subseteq \{i:[\Psi^k]_i\neq 0\} \textrm{ and } \|[\Psi^k]_{\Wcal_k}\|\geq \eta_{\Psi}\|\Psi^k\| & \textrm{  if  }\|\Psi^k\| \ge \eta \|\Phi^k\|, \label{eq.chooseW1}\\
        \Wcal_k \subseteq \{i:[\Phi^k]_i\neq 0\} \textrm{ and } \|[\Phi^k]_{\Wcal_k}\|\geq \eta_{\Phi}\|\Phi^k\| & \textrm{  if  }\|\Psi^k\| < \eta \|\Phi^k\|. \label{eq.chooseW2}
    \end{align}
\end{subequations}
with prescribed $\{\eta_\Psi,\eta_\Psi\} \in (0,1)$ and $\eta\in(0,+\infty)$. 
The first case \eqref{eq.chooseW1} indicates that the zeros  have relatively greater impact on the optimality error than the nonzeros and the algorithm solves a subproblem consisting of the zero so that $\Wcal_k \subseteq \Ical_0^k$, and vice versa.  
Obviously, $\|\Phi^k\|\neq 0 $ or $\|\Psi^k\| \neq 0$ implies that $\mathcal{W}_{k} \ne \emptyset$.

\subsection{Subspace Proximal Weighted $\ell_1$ Regularized Subproblems} \label{sec.ist}
To identify the correct support, we first solve 
a proximal-gradient subproblem for $G_k$ restricted to a subset $\Wcal$ of the variables 
\begin{equation}\label{ist} 
    \begin{aligned} 
       \underset{x_\Wcal \in \mathbb{R}^{|\Wcal|}}{ \min}\ 
       \!\langle \nabla_\Wcal f(x^k), x_\Wcal\rangle  + \tfrac{1}{2\beta } \|x_\Wcal - x^k_\Wcal \|_2^2 + \sum_{i\in \Wcal} \omega^k_i |x_i|,   
    \end{aligned}
\end{equation}
which has a closed-form solution $[ \mathbb{S}_{ \beta \omega^k }(x^k -  \beta \nabla f(x^k))]_\Wcal$. 

For sufficiently small $\beta>0$, subproblem \eqref{ist} renders a descent direction for $G_k$ and, consequently, for $F$ (as is shown in the convergence analysis). Therefore, a backtracking strategy is employed when solving \eqref{ist} to
find an appropriate $\beta$ that ensures a sufficient decrease. 
The solution of this subproblem  is detailed in subroutine  \ref{alg.IST} and is referred to as the IST (iterative soft-thresholding) step. 
For simplicity, we omit the iteration counter $k$ in the outer loop within subroutine \ref{alg.IST}. 

\begin{algorithm}[htbp]
	\caption{IST step: $[y^j,\beta_j,\Delta_{\textrm{IST}}] := \textrm{IST}(x,\omega,\Wcal; \xi_{\beta} ,\gamma_\beta,\bar{\beta})$}
         \label{alg.IST}
	\begin{algorithmic}[1]
        \REQUIRE $\{x,\omega\}\in \mathbb{R}^n$, $\Wcal \subseteq [n]$; $\{\bar{\beta},\xi_{\beta} \} \in (0,+\infty)$ and $ \gamma_\beta \in (0,1)$.
        \STATE Set $\beta_0\leftarrow \bar{\beta}$.
        \FOR {$j = 0,1,2,\ldots$} \label{IST.line  earch1}
            \STATE Set $y_{{\Wcal^c}}^{j} \leftarrow x_{{\Wcal^c}}$,   $y^{j}_{\Wcal} \leftarrow [\mathbb{S}_{\beta_{j} \omega}(x- \beta_{j} \nabla f(x))]_{\Wcal}$.\label{IST.updateX}
            \IF {$G_{k}(y^j) \leq G_{k}(x) - \frac{\xi_{\beta} }{2}\| y^j - x\|^2$}\label{IST.condTerminate}
                \RETURN $y^j$ and $\Delta_{\textrm{IST}} \leftarrow G_{k}(x) - G_{k}(y^j)$ .
            \ENDIF
            \STATE Set $\beta_{j+1} \leftarrow \gamma_\beta \beta_{j}$. \label{IST.updatemu}
        \ENDFOR \label{IST.line  earch2}
        \end{algorithmic}
\end{algorithm}

The following lemma proves that the backtracking line   search in  Algorithm \ref{alg.IST}  terminates in a finite number of steps.

\begin{lemma}\label{lem.iststepsize} Let $(x,\epsilon)\in \mathbb{R}^{n}\times  \mathbb{R}_{++}^{n}, \omega = \omega(x,\epsilon)$ and $\Wcal \subseteq [n]$. Suppose Algorithm \ref{alg.IST} is invoked with $(\hat{y}, \hat{\beta},\hat\Delta_{\textrm{IST}}) = \mathrm{IST} (x,\omega,\Wcal; \xi_{\beta} ,\gamma_\beta,\bar{\beta})$. Then Algorithm \ref{alg.IST} terminates in a finite number of iterations and satisfies
	\begin{equation}\label{add.back} F(\hat y ; \epsilon) - F(x; \epsilon) 
	\le -  \frac{\xi_{\beta} }{2}\|\hat y- x\|^2,\end{equation} 
  along with the step size bound $ \min\{\bar{\beta},\frac{\gamma_\beta}{L_1(x) + \xi_{\beta} }\} \le \hat\beta  \le \bar\beta$,  
where $L_1(x) > 0$  is the local Lipschitz constant of $\nabla f$ in a neighborhood of $x$ containing $\hat y$.  
\end{lemma}
\begin{proof}
    For any   $j$,  $y_{{\Wcal^c}}^j = x_{{\Wcal^c}}$  and $y_\Wcal^j = [ \mathbb{S}_{ \beta_j \omega }(x -  \beta_j \nabla f(x))]_\Wcal$. 
    Since 
    $x_\Wcal(\beta)  := [ \mathbb{S}_{ \beta \omega }(x -  \beta \nabla f(x))]_\Wcal$ is continuous with respect to $\beta \in [0, \bar \beta]$.
    Therefore,  $y^j$ lies on the line  segment connecting $x$  and $\bx(\hat \beta)$, on which $f$ is Lipschitz differentiable with constant $L_1(x)>0$.  
     It follows that  for any $\beta_j \le\min\{\bar\beta, \tfrac{1}{L_1(x)+\xi_{\beta} }\}$, 
        \begin{equation*}
        \begin{aligned}
            G(y^j) &\leq f(x) + \langle \nabla f(x), y^j - x\rangle + \frac{L_1(x)}{2} \|y^j - x\|^2 + \sum_{i \in \Wcal} \omega_i |y^j_i| \\
            &\leq f(x) + \langle \nabla f(x), y^j - x\rangle + \frac{1}{2\beta_j} \|y^j - x\|^2 + \sum_{i \in \Wcal} \omega_i |y^j_i| - \frac{\xi_{\beta} }{2}\|y^j - x\|^2 \\
            &\leq  f(x) + \sum_{i \in \Wcal} \omega_i |x_i| - \frac{\xi_{\beta} }{2}\|y^j - x\|^2 \\
            & = G(x) - \frac{\xi_{\beta} }{2}\|y^j - x\|^2, 
        \end{aligned}
    \end{equation*}
    where the last inequality follows from the optimality of $y^j$ for subproblem \eqref{ist}.  
   Therefore, the backtracking  terminates with $\hat \beta \ge  \min\{\bar{\beta},\frac{\gamma_\beta}{L_1(x) + \xi_{\beta} }\}$.  This further implies \eqref{add.back} as a consequence of \eqref{lem.dec0}.
\end{proof}
\subsection{Subspace Quadratic Subproblems} 



To accelerate local convergence, we solve a quadratic subproblem over a subset of nonzeros,  $\Wcal_k \subseteq \Ical^k$. The subproblem is formulated as
\begin{equation}\label{sub.qp}
	\tilde{d}^k \approx \min_{d \in \mathbb{R}^{|\Wcal_k|}}  m_k(d):= \langle g^k, d\rangle +\frac{1}{2}\langle d, H^kd\rangle,
\end{equation}
where $g^k = \nabla_{\Wcal_{k}} F(x^{k};\epsilon^{k})$ is the subspace gradient of $F(x^k; \epsilon^k)$, and $H^k = \nabla^2_{\Wcal_k\Wcal_k}F(x^k; \epsilon^k) + \zeta_k I \succ 0$ is the modified subspace Hessian with a regularization parameter $\zeta_k>0$.

Since $m^k(d)$ is defined within the subspace $\mathbb{R}^{|\Wcal_k|}$, its dimension may be small. An inexact solution $\tilde{d}^k$ is considered acceptable if it satisfies
\begin{equation}\label{eq.imposedCond}
    \langle g^k, \tilde{d}^k\rangle \leq \langle g^k, d^k_R\rangle  \textrm{ and } m_k(\tilde{d}^k) \leq m_k(0).
\end{equation}
where the reference direction $d^k_R := -\frac{\|g^k\|^2}{\langle g^k, H^k g^k \rangle} g^k$.
The subproblem can be easily solved using existing efficient quadratic programming solvers, e.g., Conjugate Gradient (CG) method. 
After (inexactly) solving the QP subproblem, we establish the following result, which is identical to \cite[Lemma 3]{chen2017reduced}. The proof is omitted for brevity.
\begin{lemma}\label{lem.dbound} 
Consider the subspace QP subproblem \eqref{sub.qp} with $H^k \succ 0$.  If $\tilde{d}^k$ satisfies \eqref{eq.imposedCond}, then the following inequalities hold 
\begin{align}
    \langle g^k, \tilde{d}^k\rangle & \leq  \langle g^k, d_R^k\rangle < 0,  \label{eq.xij2}\\ 
    | \langle g^k, \tilde{d}^k\rangle | &\ge   | \langle g^k, d_R^k\rangle |  =  \frac{\|g^k\|^2}{\langle g^k, H^k g^k\rangle} \|g^k\|^2  \geq \frac{\|g^k\|^2}{ {\lambda}_{\max}(\bH^k)}, \label{eq.xij2.2}\\
     \frac{\|g^k\|}{ {\lambda}_{\max}(\bH^k) }  & \le \|\tilde{d}^k\|\leq    \frac{2\|g^k\|}{{\lambda}_{\min}(\bH^k)}.   \label{eq.xij2.3}
 \end{align} 
\end{lemma}



\subsection{Projected Line  Search} \label{sec.proj.ls} 
Once a descent direction for $F$ is obtained by solving the QP subproblem, the projected line search subroutine (Algorithm \ref{alg.pls}) is invoked to determine a stepsize $\alpha$ that ensures a sufficient decrease in \( F(x; \epsilon) \). For clarity, we omit the iteration counter superscript $k$ in the outer loop of Algorithm \ref{alg.pls}.

Algorithm \ref{alg.pls} performs a backtracking line search along the direction \( d \) to determine a stepsize \( \alpha_j \) such that the projected update \( \Proj(x + \alpha_j d; x) \) leads to a reduction in \( F \).
When the condition \( \sign(y^j) = \sign(x) \) holds for the first time, it indicates that for any smaller stepsize $0<\alpha < \alpha_j$,  $\sign(\mathbb{P}(x+\alpha d;x)) = \sign(x)$ holds. In this case, the algorithm backtracks to check the largest stepsize \( \alpha_B := \arg\sup\{\alpha > 0: \sign(x + \alpha d) = \sign(x)\} \), producing a new iterate with the least support change compared to \( x \).

Since the algorithm is guaranteed to encounter the same sign as \( x \) within a finite number of trials, it subsequently reduces to a standard backtracking line search. Therefore, it is generally known that the overall procedure  terminates in a finite number of iterations.

\begin{algorithm}
	\caption{Projected line  search: $[y^j,\alpha_j,\Delta_{\textrm{QP}}] := \textrm{PLS}(x, \epsilon,d,\Wcal; \xi_{\alpha} , \gamma_\alpha)$}
	\label{alg.pls}
	\begin{algorithmic}[1]
		\REQUIRE $\{x,d, \epsilon\}\in \mathbb{R}^n$, $\Wcal \subseteq \mathcal{I}(x)$; $\xi_{\alpha}  \in (0,+\infty)$ and $\gamma_\alpha \in (0,1)$.
		\STATE Set $\alpha_0 \leftarrow 1$ and $\textrm{Flag} = \texttt{True}$. \label{line.proj1}
		\FOR{$j = 0,1,2,...$}
            \STATE Set $y^j\leftarrow \mathbb{P}(x+\alpha_j d;x)$.
            \IF{ $\sign(y^j) = \sign(x)$ and Flag }\label{line.alphaB}
                \STATE Set $\alpha_B \leftarrow  \arg\sup\{\alpha >0: \textrm{sgn}(x+\alpha d)= \textrm{sgn}(x)\}$ and  Flag $=$ \texttt{False}.
                \STATE Set $ \alpha_j \leftarrow \min\{1,\alpha_B\}.$
                \STATE Set $y^j\leftarrow \mathbb{P}(x+\alpha_j d;x)$.
            \ENDIF
            \IF{$F(y^j; \epsilon)\leq F(x; \epsilon) - \frac{\xi_{\alpha}}{2} \|y^j - x\|^2$} \label{line.nonincreasingF} 
		    \RETURN $y^j$ and $\Delta_{\textrm{QP}} \leftarrow F(x; \epsilon) - F(y^j; \epsilon)$.  \label{line.return1}
		    \ENDIF
            \STATE Set $\alpha_{j+1} = \gamma_\alpha\alpha_{j}$.
        \ENDFOR
	\end{algorithmic}
\end{algorithm}

{
\begin{lemma}\label{lem.SD}
Let $(x,\epsilon,d)\in \mathbb{R}^{n} \times \mathbb{R}_{++}^{n} \times \mathbb{R}^{|\Wcal|}$, $\Wcal \subseteq \Ical(x)$ such that $\langle \nabla_{\Wcal} F(x;\epsilon), d\rangle < 0$. Suppose Algorithm \ref{alg.pls} is invoked with $(\tilde{y},\tilde{\alpha},\tilde \Delta_{\textrm{QP}}) =  \mathrm{PLS}(x, \epsilon,d,\Wcal; \xi_{\alpha} , \gamma_\alpha)$. 
Then Algorithm \ref{alg.pls} terminates finitely satisfying
\begin{equation}\label{pls.des}
    F(\tilde{y};\epsilon) - F(x;\epsilon) \leq -\frac{\xi_{\alpha}}{2}\|\tilde{y}-x\|^2.
\end{equation}
Moreover, if $\textrm{sgn}(\tilde{y}) \ne \textrm{sgn}(x)$, then
\begin{equation}\label{reduce.support.1} 
	 1 \geq \tilde{\alpha} \geq \min\{\alpha_B,1\} \  \textrm{ and }\  \Ical(\tilde{y}) \subset \Ical(x).
\end{equation}
Otherwise, $F(x; \epsilon)$ is  $L_2(x;\epsilon)$-Lipschitz differentiable on the  line segment $[x,x+\alpha_B d)$ and
\begin{equation}\label{reduce.support.2.orig} 
    \min\{1,\alpha_B\} \geq \tilde{\alpha}  \ge  \min\{1, \frac{-2 \gamma_\alpha\langle \nabla_{\Wcal} F(x;\epsilon), d\rangle }{ (L_2(x;\epsilon)+\xi_{\alpha})\|  d \|^2} \}\  \textrm{ and }\  \Ical(\tilde{y}) = \Ical(x).
    \end{equation} 
\end{lemma}
}
\begin{proof} 
We restrict the analysis to $\mathbb{R}^{|\Wcal|}$, since $d_{\Wcal^c} = 0$, and omit the subscript $\Wcal$ for brevity.
If $\textrm{sgn}(\tilde{y}) \ne \textrm{sgn}(x )$, then naturally \eqref{pls.des}  holds and $\tilde{y}$ is on the boundary of the subspace orthant containing $x$, which means $\Ical(\tilde{y})  \subset  \Ical(x)$. 
If $\textrm{sgn}(\tilde y ) = \textrm{sgn}(x )$, then line \ref{line.alphaB} is  triggered and there are no points of nondifferentiability of $F(x;\epsilon)$ in subspace $\mathbb{R}^{|\Wcal|}$ in the line segment connecting $x $ to $x  + \alpha_B d$. We end up with a traditional backtracking line  search with $F(x; \epsilon)$ being $L_2(x;\epsilon)$-Lipschitz differentiable on the  line segment $[x,x+\alpha_B d)$,
 \begin{equation*}
 \begin{aligned}
     F(x + \alpha_j d; \epsilon) - F(x; \epsilon ) &\leq \alpha_j\langle  \nabla F(x; \epsilon ),  d \rangle + \frac{1}{2} L_2(x;\epsilon)  \| \alpha_j d \|^2\\
     & \leq -\frac{\xi_{\alpha}}{2} \|\alpha_j d\|^2  = -\frac{\xi_{\alpha}}{2} \|y^j - x\|^2,
 \end{aligned}
\end{equation*} 
where the second inequality is satisfied for any $\alpha_j \in \bigg(0, \frac{-2\langle \nabla F(x;\epsilon), d\rangle }{ (\xi_{\alpha} + L_2(x;\epsilon))\|  d \|^2} \bigg].$ Therefore, the line  search terminates with $\tilde{\alpha}$   satisfying \eqref{reduce.support.2.orig} and $\Ical(\tilde{y})  =  \Ical(x)$.    
\end{proof} 

\subsection{Main Algorithm}

{
With the two subproblems clearly defined, we now formalize the main algorithm in Algorithm \ref{algo.IReNA}. At each iteration, we first solve the IST subproblem to obtain $\hat x^k$. The condition $\sign(\hat x^k) = \sign(x^k)$ in line \ref{line.iflocal} serves as a potential indicator of stability before proceeding to the QP subproblem.

By the construction of $\Psi^k$ and $\Phi^k$, the selected subset $\Wcal_k$ is guaranteed to satisfy one of the following cases: $\Wcal_k \subseteq \Ical_0^k$, as determined by \eqref{eq.chooseW1}; $\Wcal_k \subseteq \Ical^k$, as determined by \eqref{eq.chooseW2}. 
If $\Wcal_k \subseteq \Ical_0^k$, then $\hat x^k = x^k + d(\hat \beta_k) = d(\hat \beta_k) \neq 0$ by \eqref{le.dk1} , the IST subproblem  induces sign changes, leading to $\Ical^k \subset \Ical(\hat x^k)$. Consequently, line \ref{line.iflocal} always evaluates to ``\texttt{False}", preventing the QP subproblem from being triggered.
On the other hand, if $\Wcal_k \subseteq \Ical^k$, we have $\Ical(\hat x^k) \subseteq \Ical^k$, ensuring that the QP subproblem, when triggered, operates exclusively on a subset of the nonzero components.

It is also important to note that in line \ref{line.badnt} of Algorithm \ref{algo.IReNA}, we compare the decrease $ \tilde \Delta_{\textrm{QP}}$ resulting from the QP subproblem with the decrease $\hat\Delta_{\textrm{IST}}$ obtained from the IST subproblem.  If $\tilde\Delta_{\textrm{QP}} < \nu \hat \Delta_{\textrm{IST}}$ for a prescribed $\nu > 0$, the algorithm prefers the iterate generated by the IST subproblem. This mechanism ensures that the QP solution is only adopted when it provides an effective reduction in the objective function and the sign is not locally stabilized.
}

 \begin{algorithm}
	\caption{Proposed IReNA for solving \eqref{p.prob}} \label{algo.IReNA} 
	\begin{algorithmic}[1]
	\REQUIRE $(x^0,\epsilon^0) \in \mathbb{R}^{n} \times \mathbb{R}^{n}_{++}$, $\{\eta_{\Phi}, \eta_{\Psi}\} \in (0,1]$, $\{\gamma_\epsilon, \gamma_\beta, \gamma_\alpha\} \in (0,1)$, $\{\tau ,\eta,\xi_{\beta} ,\xi_{\alpha},\bar{\beta},  \zeta_k, \nu \} \in (0, \infty)$.
	\FOR{$k=0, 1, 2, \ldots$}
	        \STATE (\textbf{Optimality measure}) Compute $\Psi^k, \Phi^k$.
	\IF{$\max\{ \|\Psi^k\|, \|\Phi^k\|\} \leq\tau$ and $\epsilon^k_i \le \tau , i\in\Ical^k$} \label{line.terminate}
		\RETURN{the (approximate) solution $x^k$. 
		} 
		\label{line.return}
		\ENDIF
        \STATE (\textbf{Subspace determination}) Choose $\Wcal_k$ based on $\Psi^k,\Phi^k$ by \eqref{eq.chooseW}. \label{line.Wk}
        \STATE (\textbf{IST step \ref{alg.IST}}) Compute $(\hat x^{k}, \hat \beta_{k},\hat \Delta^k_{\textrm{IST}}) \leftarrow \textrm{IST}(x^k,\omega^k,\Wcal_k;\xi_{\beta} ,\gamma_\beta,\bar{\beta})$.\label{line.IST}	
        \IF{ $\textrm{sgn}(\hat x^{k} ) = \textrm{sgn}(x^k)$}\label{line.iflocal}
        \STATE (\textbf{QP subproblem}) Find an approximate solution $\tilde{d}^k$ to \eqref{sub.qp} satisfying \eqref{eq.imposedCond}.\label{line.QP}
        \STATE (\textbf{Line search \ref{alg.pls}}) Compute $( \tilde x^k,\tilde \alpha_k,\tilde \Delta^k_{\textrm{QP}}) \leftarrow \textrm{PLS}(x^k, \epsilon^k,\tilde{d}^{k},\Wcal_k;\xi_{\alpha} ,\gamma_\alpha)$. \label{line.nt} 
        \IF{$\textrm{sgn}(\tilde x^k ) \neq \textrm{sgn}(x^k)$ and $\tilde\Delta^k_{\textrm{QP}} < \nu\hat\Delta^k_{\textrm{IST}}$} \label{line.badnt}
           \STATE {Set $x^{k+1} \leftarrow \hat x^k$.} \label{line.ist1}
           \ELSE 
            \STATE Set $x^{k+1} \leftarrow \tilde x^k$. \label{line.qp2}
		\ENDIF
        \ELSE
        \STATE Set $x^{k+1} \leftarrow \hat x^{k}$. \label{line.ist2}
        \ENDIF
	    \STATE Update  $\epsilon_i^{k+1}  \in(0, \gamma_\epsilon\epsilon^k_i)$, $i\in\Ical_0^{k+1}$ and $\epsilon_i^{k+1} = \epsilon_i^{k}$, $i\in\Ical_0^{k+1}$.  \label{line.eps}
        \STATE Update {$\omega_i^{k+1} = \lambda  r'(|x_i^{k+1}|+\epsilon^{k+1}_i)$}, $i \in [n]$.\label{line.reweight}
	\ENDFOR
	\end{algorithmic}
\end{algorithm}


Proposition \ref{prop:optimility} implies that the algorithm converges to an optimal solution of $F(x)$ if $\max\{ \|\Psi^k\|, \|\Phi^k\|\} \to 0$ and $\epsilon^k_i \to 0$  for all $i\in\Ical^k$. 
This criterion underlies our termination condition (line \ref{line.terminate}--line  \ref{line.return}). 
To ensure convergence, the updating strategy for $\epsilon$ should be carefully designed. Specifically, once the true support $\Ical(\bx^*)$ is identified, we drive $\epsilon^k_i$ to zero rapidly for all $i\in\Ical(\bx^*)$, while keeping $\epsilon^k_i$ fixed for $i\in\Ical_0(\bx^*)$.  By doing so, we can effectively eliminate the zeros and their associated $\epsilon_i$ from $F(x; \epsilon)$, making the problem resemble a smooth optimization
problem---a crucial aspect of the convergence analysis. To achieve this, our updating strategy (line \ref{line.reweight}) reduces $\epsilon_i$ only for indices in $i\in\Ical^k$. This dynamic strategy, originally proposed in \cite{wang2021relating}, was proven to stop updating $\epsilon_i$ for $i\in\Ical_0(\bx^*)$ while consecutively decreasing  $\epsilon_i, i\in\Ical(\bx^*)$ for sufficiently large iterations.

\section{Convergence Analysis}\label{sec.convergence}

We need the following assumptions in the analysis.  

\begin{assumption}\label{assumption.Hk}
The tolerance is set as $\tau  = 0$ and  the level set $\textrm{Lev}_{F} :=\{x:F(x)\leq F(x^0; \epsilon^0) \}$ is contained in a bounded ball $\{x \mid \|x\|\le R\}$ so that $F$ is bounded below and $f$ is Lipschitz differentiable   on $\textrm{Lev}_{F}$  with constant $L_f > 0$. 
  \end{assumption}

%

We can use the results in \S\ref{sec.algorithm} to establish the well-posedness of IReNA. 
\begin{theorem}[Well-posedness]\label{well.pose}  
    Suppose $\tau  = 0$. Then Algorithm \ref{algo.IReNA}   produces an infinite sequence $\{(x^k,\epsilon^k)\}$ satisfying that $\{F(x^k;\epsilon^k)\}$ and $\{\epsilon_i^k\}, i\in[n]$  are both monotonically decreasing  and convergent with $\{\bx^k\} \subset \textrm{Lev}_{F}$. 
\end{theorem}
\begin{proof} Each $\{\epsilon_i^k\}$ is obviously monotonically decreasing.   Combining \eqref{pls.des} and \eqref{add.back} gives $F(x^{k+1};\epsilon^{k+1}) \le F(x^k;\epsilon^k)$. 
Moreover, it follows from 
$F(\bx^k) < F(\bx^k;\epsilon^k)\le F(\bx^0;\epsilon^0)$ that $\{x^{k}\} \subset \textrm{Lev}_{F}$.
\end{proof}

{ 
Theorem \ref{well.pose} implies that the local Lipschitz constant of $f$ in Lemma \ref{lem.iststepsize} satisfies $L_1(x^k) \le L_f$, so that 
\begin{equation}\label{mu is bound} \underline \beta \le \hat \beta_k\le \bar \beta\end{equation}
with $\underline\beta:=\min\{\bar{\beta},\frac{\gamma_\beta}{L_f + \xi_{\beta} }\}$.  
Next, we show the same results for $ \lambda_{\min}(\bH^k)$, $\lambda_{\max}(\bH^k)$  in Lemma \ref{lem.dbound} and $L_2(x^k;\epsilon^k)$ in  Lemma \ref{lem.SD}. 

\begin{lemma}\label{lem.H}
    Let $\{x^k\}$ be the sequence generated by Algorithm \ref{algo.IReNA}.   
    Consider all $\bH^k$ generated in subproblem \eqref{sub.qp}. 
    There exists $0 < \lambda_{\min} <\lambda_{\max}<+\infty$ such that  
    $ \lambda_{\min} \bI \preceq    \bH^k  \preceq  \lambda_{\max} \bI $. Moreover,  there exists $L_F>0$ such that 
    $L_2(x^k;\epsilon^k) < L_F$. 
\end{lemma}
\begin{proof}
Since the eigenvalues of $\bH^k$ and $L_2(\bx^k;\epsilon^k)$ all depend on $r''(|x_j^k|), j\in\Ical^k$ and $x^k$ is bounded in $\textrm{Lev}_{F}$, 
we only need to examine the case when $\liminf\limits_{k \to +\infty} |x_j^k| = 0$ for some $j \in [n]$. 
    If $r'(0^+) <  \infty$,  then $|r''(0^+)| < \infty$ and 
    the statement naturally holds. 

    Now consider the case $r'(0^+) =  \infty$, which implying $r''(0^+) = -\infty$ by concavity. We claim that there exists $\delta > 0$ such that $|x^k_i| > \delta, i \in \Ical^k$ holds true for all $k$.  
Suppose by contradiction that this is not true. Then there exist   $\Scal\in\mathbb{N}$ and an index $j$ such that 
\begin{equation}\label{eq.contrad} 
   j \in\mathcal{I}(x^{k}), \  k\in \Scal, \  |\Scal | = +\infty
    \end{equation} 
      and $\{|x_j^k|\}_{\Scal}$ is strictly decreasing to 0. 

This means  
$  {\epsilon}^k_j \to 0$ and $ \omega_j^k \overset{\Scal}{\to} \lambda r'(0^{+}) = +\infty$. 
Therefore,  there exist a sufficiently large $\check{k}  \in \Scal$ such that 
 \begin{equation}\label{extreme.bound}
   \hat\beta_{\check{k}} \omega_j^{\check{k}} \ge \underline\beta \omega_j^{\check{k}}  > \sup_{\bx\in \textrm{Lev}_F} \{ |x_j| + \bar\beta |\nabla_jf(\bx)|\} > | x_j^{\check{k}} - \hat\beta_{\check{k}} \nabla_j f(x^{\check{k}})|.
   \end{equation}
By Lemma \ref{lem.ist.ist}, we know that 
 $d_j(\hat\beta_{\check{k}}) = -x_j^{\check{k}}$. In other words, the IST subproblem returns  
 $\hat x_j^{\check{k}+1} = 0$.  This means $\textrm{sgn}(\hat{x}^{\check{k}+1}) \ne \textrm{sgn}(x^{\check{k}})$, so that the QP subproblem is not triggered and 
 $x^{\check{k}+1} =  \hat{x}^{\check{k}+1}$, resulting in $j \in \Ical_0^{\check k+1}$. 

Since $0= |x_j^{\check{k} + 1}| < x_j^{\check{k}}$ and $\epsilon_j^{\check{k}+1} = \epsilon_j^{\check{k} } $, we have    $\omega_j^{\check{k} + 1} > \omega_j^{\check{k}  } $, which implies that    \eqref{extreme.bound} also holds at $\check{k} + 1$, or equivalently, 
$\omega_j^{\check{k}+1} > | \nabla_j f(\bx^{\check{k}+1})|$. By the definition of $\Psi$,  
$[\Psi (\bx^{\check{k} +1};\epsilon^{\check{k} + 1})]_j= 0$. This means $j$ is not selected by line \ref{line.Wk}, so that  $x_j^{\check{k} + 2} = 0$. We can repeat the same argument for $k\ge \check{k} +2$ and obtain that   $x_j^k \equiv 0$, which contradicts \eqref{eq.contrad}.  The proof is complete.
\end{proof}

Lemma \ref{lem.H} implies that the lower bound described in \eqref{reduce.support.2.orig} can be replaced by 
\begin{equation}\label{alpha is bound} \underline \alpha \le \tilde\alpha_k\le 1\end{equation}
with $\underline \alpha := \frac{\gamma_\alpha\lambda^2_{\min}}{2\lambda_{\max}L_F+2\xi_{\alpha}\lambda_{\max}}$. The lower bound is established from \eqref{reduce.support.2.orig}:
\begin{equation*}
    \tilde\alpha_k  \ge   \frac{-2 \gamma_\alpha\langle \nabla_{\Wcal} F(x^k;\epsilon^k), d^k\rangle }{ (L_2(x^k;\epsilon^k)+\xi_{\alpha})\|  d^k \|^2} 
    \geq 2\frac{\gamma_\alpha\frac{\|g^k\|^2}{\lambda_{\max}}}{(L_2(x^k;\epsilon^k)+\xi_{\alpha}) \frac{4 \|g^k\|^2}{\lambda^2_{\min}}} \geq \frac{\gamma_\alpha\lambda^2_{\min}}{2\lambda_{\max} L_F+2\xi_{\alpha}\lambda_{\max}},
\end{equation*}
{where the second inequality is from \eqref{eq.xij2.2} and \eqref{eq.xij2.3} and the third inequality uses Lemma \ref{lem.H}.}
}


\subsection{Global Convergence} 
The global convergence of IReNA is established in this subsection. For ease of presentation, we first define the following sets of iterations for our analysis.
$$
\begin{aligned}
\Scal_{\textrm{IST}} &:= \{k \in \mathbb{N}: \textrm{$x^{k+1} = \hat x^{k+1}$ by line  \ref{line.ist1} or \ref{line.ist2} at the $k$th iteration}\},\\ 
\Scal_{\textrm{QP}} &:= \{k \in \mathbb{N}: \textrm{$x^{k+1} = \tilde x^{k+1}$ by line  \ref{line.qp2} at the $k$th iteration}\}.\\ 
\end{aligned}
$$
The set $\Scal_{\textrm{QP}} $ is further spitted into two subsets:
$$
\Scal_{\textrm{QP}}^{\textrm{N}}:=\{k\in \Scal_{\textrm{QP}} :\textrm{sgn}(x^{k+1})\neq\textrm{sgn}(x^{k})\}, \quad
\Scal_{\textrm{QP}}^{\textrm{Y}}:= \Scal_{\textrm{QP}}\setminus \Scal_{\textrm{QP}}^{\textrm{N}}.
$$

We are now ready to prove the global convergence of the proposed algorithm. 
{ 
\begin{theorem}[Global convergence]\label{theo.global.conv}   
Let $\{(x^k,\epsilon^k)\}_{k \in \mathbb{N}}$ be the sequence generated by Algorithm \ref{algo.IReNA}. Then 
\[
\lim_{k \to +\infty} \|x^{k+1}-x^k\| \;=\; 0 
\quad\text{and}\quad
\lim_{k \to +\infty} \|\epsilon^{k+1}-\epsilon^k\| \;=\; 0.
\]
Moreover,  it holds that
\[
\lim_{k \to +\infty} \|\Phi^k\| \;=\; 0 
\quad\text{and}\quad
\lim_{k \to +\infty} \|\Psi^k\| \;=\; 0.
\]
\end{theorem} 
Let $\chi^{\infty}$ be the set of all cluster points of $\{x^k\}$. Then any $x^* \in \chi^{\infty}$ satisfies the first-order stationary condition \eqref{eq.optimalcondition}.
\begin{proof}

From \eqref{add.back}, \eqref{pls.des}, we have
\begin{equation*}
    \sum_{k=1}^{t} \Bigl[F(x^k;\epsilon^k) \;-\; F(x^{k+1};\epsilon^{k+1}) \Bigr]  
     \geq  
    \sum_{k \in \Scal_{\textrm{IST}}}^{t} \frac{\xi_{\beta}}{2} \|x^{k+1}-x^k\|^2 
    \;+\;
    \sum_{k \in \Scal_{\textrm{QP}}}^{t} \frac{\xi_{\alpha}}{2} \|x^{k+1}-x^k\|^2.
\end{equation*}
Letting $t \to +\infty$, it follows immediately that 
\[
\sum_{k=1}^{\infty} \|x^{k+1}-x^k\|^2 < +\infty.
\]
Hence $\|x^{k+1}-x^k\| \to 0$ as $k \to +\infty$. From line~\ref{line.eps} of Algorithm~\ref{algo.IReNA}, we also obtain $\|\epsilon^{k+1}-\epsilon^k\| \to 0$. 

We now consider the three sets of iterations and show that the reduction in objective is lower bounded by the optimal residuals:
\begin{itemize}
    \item[(i)] For $k \in \Scal_{\textrm{IST}}$, if $\Wcal_k$ is chosen from \eqref{eq.chooseW1}, then we have 
    \begin{equation*}
        \begin{aligned}
        F(x^{k};\epsilon^{k}) - F(x^{k+1};\epsilon^{k+1})
        & \geq
         \frac{\xi_{\beta} }{2} \|x^{k+1}-x^{k}\|^2 
        \geq
         \frac{\xi_{\beta} }{2}\|\underline\beta \,\Psi^k_{\Wcal_k}\|^2 \\
        &\geq
         \frac{1}{2}\xi_{\beta}  \,\eta_{\Psi}^2\,\underline\beta^2\,\|\Psi^k \|^2
        \geq c_1 \max\{\eta^2 \|\Phi^k\|^2, \|\Psi^k\|^2\},  
        \end{aligned}
    \end{equation*}
    where the first inequality follows from \eqref{add.back} and \eqref{lem.dec0}, the second from \eqref{mu is bound} and \eqref{le.dk1}, and the last two hold for \eqref{eq.chooseW1} with $c_1 := \frac{1}{2}\xi_{\beta}  \,\eta_{\Psi}^2\,\underline\beta^2$. 
    Similarly, if $\Wcal_k$ is chosen from \eqref{eq.chooseW2}, then 
     by \eqref{add.back}, \eqref{mu is bound}, \eqref{le.dk2} and \eqref{eq.chooseW2}, we have
     \begin{equation}\label{phi_dec}
        \begin{aligned}
           F(x^{k};\epsilon^{k}) - F(x^{k+1};\epsilon^{k+1}) 
           &\geq \frac{\xi_{\beta} }{2} \|x^{k+1}-x^{k}\|^2 
         \geq  \frac{\xi_{\beta} }{2}\|\min\{\underline\beta,1\} \,\Phi^k_{\Wcal_k}\|^2 \\
         &\geq \frac{1}{2}\xi_{\beta}  \,\eta_{\Phi}^2\,\min\{\underline\beta^2,1\}\,\|\Phi^k \|^2 \geq c_2 \max\{\eta^2 \|\Phi^k\|^2, \|\Psi^k\|^2\} ,
        \end{aligned}
    \end{equation}
    where $c_2 := \frac{1}{2}\xi_{\beta}  \,\eta_{\Phi}^2\,\min\{\underline\beta^2,1\}\, \eta^{-2}$.
    \item[(ii)] For $k \in \Scal_{\textrm{QP}}^{\textrm{N}}$, then $\Wcal_k$ is chosen from \eqref{eq.chooseW2} and line \ref{line.badnt} does not hold. 
    We have from \eqref{pls.des} that 
    \begin{equation*}
         \begin{aligned}
            F(x^{k};\epsilon^{k}) - F(x^{k+1};\epsilon^{k+1}) \geq \tilde\Delta^k_{\textrm{QP}}
          \geq  \nu \hat\Delta^k_{\textrm{IST}} \geq  \nu c_2 \max\{\eta^2 \|\Phi^k\|^2, \|\Psi^k\|^2\},
         \end{aligned}
    \end{equation*}
    where the last inequality follows from
    \begin{equation*}
         \hat\Delta^k_{\textrm{IST}} = G_k(x^{k};\epsilon^{k}) - G_k(\hat x^{k+1}) \geq c_2 \max\{\eta^2 \|\Phi^k\|^2, \|\Psi^k\|^2\},  
    \end{equation*}
    where the inequality follows similarly to \eqref{phi_dec}.
    \item[(iii)] For $k \in \Scal_{\textrm{QP}}^{\textrm{Y}}$, then $\Wcal_k$ is chosen from \eqref{eq.chooseW2}. We have from \eqref{pls.des} and \eqref{lem.dec0} that
    \begin{equation*}\label{eq.sd2}
         \begin{aligned}
            F(x^{k};\epsilon^{k}) - F(x^{k+1};\epsilon^{k+1})
            &  \geq  \frac{\xi_{\alpha}}{2} \|x^{k+1}-x^{k}\|^2 
             \geq \frac{\xi_{\alpha}}{2}\| \underline\alpha\,d^k_{\Wcal_k}\|^2 
             \geq \frac{\xi_{\alpha}\,\underline{\alpha}^2}{2\,\lambda^2_{\max}} \|g^k\|^2\\
            & \geq \frac{\xi_{\alpha}\,\underline{\alpha}^2}{2\,\lambda^2_{\max}} \|\Phi^k_{\Wcal_k}\|^2
             \geq \frac{\xi_{\alpha}\,\underline{\alpha}^2\,\eta^2_{\Phi}}{2\,\lambda^2_{\max}} \|\Phi^k\|^2 \\
             &\geq c_3 \max\{\eta^2 \|\Phi^k\|^2, \|\Psi^k\|^2\}.
         \end{aligned}
    \end{equation*}
    where the second inequality follows from \eqref{alpha is bound}, the third from \eqref{eq.xij2.3}, the fourth from Proposition \ref{prop:optimility}(i), and the last two hold for \eqref{eq.chooseW2} with $c_3:= \frac{\xi_{\alpha}\,\underline{\alpha}^2\,\eta^2_{\Phi}\eta^{-2}}{2\,\lambda^2_{\max}}$.  
\end{itemize}
Summing over these inequalities and letting $k \to +\infty$ gives
\begin{equation*}
        \sum_{k=1}^{\infty} \Bigl[F(x^k;\epsilon^k) \;-\; F(x^{k+1};\epsilon^{k+1}) \Bigr] \geq
         \sum_{k=1}^{\infty} \min\{c_1,c_2,\nu c_2,c_3\}\max\{\eta^2 \|\Phi^k\|^2, \|\Psi^k\|^2\} .
\end{equation*}
Using the fact that $F$ is bounded below on $\mathrm{Lev}_F$ (by Assumption~\ref{assumption.Hk}), we have $\|\Phi^k\| \to 0$ and $\|\Psi^k\| \to 0$. 

Consider any \( x^* \in \chi^{\infty} \) and a subsequence \( \{x^k\}_{\Scal} \) such that \( x^k \to x^* \). Then, there exists a sufficiently large \( \check{k} \in \Scal \) such that for all \( k \in \Scal \) with \( k > \check{k} \), we have \( x^k_j \in \Ical^k \) for all \( j \in \Ical(x^*) \). This, together with the \( \epsilon \) update strategy in line \ref{line.eps}, ensures that  
\[
\lim_{k\to +\infty, k\in\Scal} \epsilon^k_j = 0, \quad \forall j \in \Ical(x^*).
\]  
By Proposition \ref{prop:optimility}(iii), it follows that \( x^* \) satisfies \eqref{eq.optimalcondition}. This completes the proof.
\end{proof}

}





{ 
Next we show that the iterates generated by IReNA possess the locally stable sign property after a finite number of iterations, which means that $\sign(x^k)$ remains unchanged after some finite $k\in \mathbb{N}$. 
Establishing sign stability requires identifying the nonzero components within a finite number of iterations. This problem is closely related to the topic of active manifold in nonsmooth  optimization, which typically requires the  nondegeneracy condition $0 \in \textrm{rint } \partial f(x)$ to hold at a critical point \cite{liang2017activity, hare2009proximal}. 
We also impose this assumption, stated equivalently as follows:
\begin{assumption}\label{ass.inter}
At any stationary point $x^{*}$ of \eqref{p.prob},  
\[  \nabla_i f(x^{*}) \in (-\lambda r'(0^+),\lambda r'(0^+)), \quad i\in\Ical_0(\bx^*).\]
\end{assumption}
This assumption naturally holds when $r'(0^+)=+\infty$. With this, we are now ready to establish the local support stability of the iterates.

\begin{proposition}[Locally stable sign]\label{Prop.support.stable}
 Let $\{x^k\}$ be the sequence generated by Algorithm \ref{algo.IReNA}. The following statements hold. 
    \begin{itemize}
        \item[(i)] There exists $\delta > 0$ such that 
        $ |x^k_i| > \delta, i \in \Ical^k$ holds true for all     $k$. Consequently, { $\omega^k_i < \hat{\omega} :=\lambda r'(\delta), i\in\Ical^k$} holds true for all $k$. 
        \item[(ii)] There exist index sets $ \Ical_0^* $ and $\Ical^*$ such that $\Ical_0^k\equiv  \Ical_0^*$ and $\Ical^k \equiv \Ical^*$ for sufficiently large $k$. 
        \item[(iii)] $k\in \Scal^{\textrm{Y}}_{\textrm{QP}}$ and $\sign(x^{k+1}) = \sign(x^k)$ for sufficiently large $k$.
    \end{itemize}
\end{proposition} 

\begin{proof} 
(i) From the proof of Lemma \ref{lem.H}, we have already established this statement when $r''(0^+) = -\infty$ and $r'(0^+) = +\infty$. Using the same argument, we can prove the case when $r'(0^+) < +\infty$. 
{ Suppose by contradiction that there exist $\Scal\in\mathbb{N}$ and an index $j$ such that 
\begin{equation}\label{eq.contrad2}
   j \in\mathcal{I}(x^{k}), k\in \Scal, \  |\Scal | = +\infty, \{|x_j^k|\}_{\Scal} \to 0.
    \end{equation} 
This means $  {\epsilon}^k_j \to 0$, $ \omega_j^k \overset{\Scal}{\to} \lambda r'(0^{+})$. 
Let \( x^* \) be a cluster point of \( \{x^k\}_{\check\Scal} \) with \( \check\Scal \subseteq \Scal \). Then, from Theorem \ref{theo.global.conv}, we know that \( x^* \) is a first-order stationary point. By Assumption \ref{ass.inter} and \eqref{eq.contrad2}, it follows that  
$\lim_{k \to \infty, k \in \check\Scal} \left| \nabla_j f(x^{ k}) - x_j^{k}/{\hat\beta_{k}} \right| = |\nabla_j f(x^*)| < |\lambda r'(0^+)|$  
with \( \hat\beta_{k} \) being bounded.
Therefore, we know there exists a sufficiently large $\check{k} \in \check\Scal$ such that
\begin{equation}\label{extreme.bound2}
    \hat\beta_{\check{k}}\omega_j^{\check{k}} > |x_j^{\check{k}} - \hat\beta_{\check{k}}\nabla_j f(x^{\check{k}}) | .
\end{equation}
The steps from \eqref{extreme.bound2} onward are analogous to those in the proof of Lemma \ref{lem.H} (see steps following \eqref{extreme.bound}), resulting in a contradiction. The proof is complete
}
 
(ii) Suppose by contradiction that this is not true.  Then there exists $j$ and $\mathbb{N} = \Scal \cup \Scal_0$ such that 
 $ |\Scal | = +\infty,  |\Scal_0 | = +\infty, x_j^k \in \Ical^k, k\in\Scal$ and $ x_j^k \in \Ical_0^k, k\in\Scal_0$. 
 It then follows from the update strategy (\ref{line.reweight})  of $\epsilon$ 
 that $\epsilon_j^k$ is reduced for all $k\in \Scal$ and therefore $\epsilon_j^k \to 0$. 
 Now for sufficiently large $k  \in\Scal_0$ satisfying   $\omega_j^k > L_f$ and $\Psi^k_j = 0$. 
 Therefore, $j$ will never be chosen by   \eqref{eq.chooseW2} and will stay in $\Ical_0^k$ for $\{k+1, k+2, ...\}$, which means $|\Scal| < +\infty$---a contradiction.  

(iii)
We only need to verify that \( k \in \Scal^{\textrm{Y}}_{\textrm{QP}} \) for all sufficiently large \( k \), as \( \sign(x^{k+1}) = \sign(x^k) \) naturally follows if the statement holds. Since statement (ii) holds, for sufficiently large \( k \), we have \( k \notin \Scal_{\textrm{QP}}^{\textrm{N}} \), and \( x^{k+1} \) cannot be obtained from line \ref{line.ist2}. It suffices to show that \( x^{k+1} \) is not obtained from line \ref{line.ist1} infinitely often.  
Suppose, for contradiction, that there exists an infinite subsequence \( \{x^k\}_{\Scal} \) with \( |\Scal| = +\infty \) such that \( x^{k+1} = \hat x^{k+1} \) is obtained from line \ref{line.ist1}. Then, for sufficiently large \( k \in \Scal \), we have  
\begin{equation}
    G_k(x^k) - G_k(x^{k+1}) = \hat\Delta^k_{\textrm{IST}} >\frac{1}{\nu}\tilde\Delta^k_{\textrm{QP}} \geq \frac{\xi_{\alpha}}{2\nu} \|\tilde x^{k}-x^k\|^2 \geq \frac{\xi_{\alpha} \delta^2}{2\nu},  
\end{equation}  
where the last inequality follows from \( \sign(\tilde x^k) \neq \sign(x^k) \) and statement (i). This implies that the objective decrement \( \hat \Delta^k_{\textrm{IST}} \) remains bounded below infinitely often, contradicting the boundedness of \( F \) in Assumption \ref{assumption.Hk}. The proof is complete.
\end{proof} 

Proposition \ref{Prop.support.stable}(ii) implies that there exists $\check{k}$ such that for all $k\ge \check{k}$,   $(x^k,\epsilon^k)$ always stays in the reduced subspace 
$ \Mcal(x^*,\epsilon^*):=\{(x,\epsilon) \mid x_{\Ical_0^*} = 0, \epsilon_{\Ical_0^*} \equiv \epsilon_{\Ical_0^*}^{\check{k}} \}$
and $x^k$ is contained in the reduced subspace
$\overline  \Mcal(x^*):=\{x \mid x_{\Ical_0^*} = 0 \}.$  

The local equivalence between $\Phi$ and the subspace gradient of $F(x;\epsilon)$ also follows from Proposition \ref{prop:optimility}, as shown in the following corollary.
}
\begin{corollary}\label{coro.grad}  
Let $\{(x^k,\epsilon^k)\}_{k \in \mathbb{N}}$ be the sequence generated by Algorithm \ref{algo.IReNA}. Then $[\Phi(x^k;\epsilon^k)]_i  = \nabla_iF(x^k;\epsilon^k), i\in\Ical^*$ for sufficiently large $k$.  Consequently, $\lim_{k\to+\infty}$ $\nabla_{\Ical^*}F(x^k;\epsilon^k) = 0.$
\end{corollary}  
 \begin{proof} 
 It suffices to prove the second and third cases in \eqref{eq:nonzerosmeasure}.  For $i\in\Ical^k_+$ and $\nabla_if(x^k)+\omega^k_i > 0$, Proposition \ref{Prop.support.stable} implies 
 that $\max\{x^k_i, \nabla_i f(x^k) - \omega^k_i\} > \delta > 0$. Note that $[\Phi^k]_i \to 0$ by Theorem \ref{theo.global.conv}. This implies that $[\Phi^k]_i = \nabla_i f(x^k)+\omega^k_i = \nabla_i F(x^k;\epsilon^k)$ for all sufficiently large $k$. For $i\in\Ical_-^k$ and $\nabla_if(x^k)-\omega^k_i < 0$, similar arguments establish $[\Phi(x^k;\epsilon^k)]_i = \nabla_i f(x^k)+\omega^k_i = \nabla_i F(x^k;\epsilon^k)$ for all sufficiently large $k$.    
\end{proof} 
 
\subsection{Convergence Rate Under the KL Property}
In this subsection, we establish the local convergence properties by using the well-known KL property.  For clarity, we adopt the definitions of the KL exponent from \cite[Definitions 2.3]{li2018calculus}.

\begin{definition}[KL exponent]   For a proper closed function $f$ satisfying the KL property at $\bar{x} \in \text{dom } \partial f$, if the corresponding function $\phi$ can be chosen as $\phi(s)=cs^{1-\theta}$ for some $c>0$ and $\theta \in [0,1)$,  i.e., there exist $c, \rho > 0$ and $\upsilon \in (0, \infty]$ so that
$$\text{dist}(0,\partial f(x)) \geq c(f(x)-f(\bar{x}))^{\theta}$$
whenever $x \in \Bcal(\bar x,\rho) $ and $f(\bar{x}) < f(x) < f(\bar{x}) + \upsilon$, then we say that $f$ has the KL property at $\bar{x}$ with an exponent of $\theta$. If $f$ is a KL function and has the same exponent $\theta$ at any $\bar{x} \in \text{dom } \partial f$, then we say that $f$ is a KL function with an exponent of $\theta$.
\end{definition}

Based on the above results, we restrict our discussion to the reduced subspace $\mathbb{R}^{\vert \Ical^*\vert}$. In addition, we denote $\epsilon = \varepsilon \circ \varepsilon$ by $\varepsilon \geq 0$ and treat $\varepsilon$ as a variable. 
As noted in \cite[Page 63, Section 2.1]{luo1996mathematical}, determining 
or estimating the KL exponent of a given function is often extremely challenging. A particularly relevant and useful result is the theorem provided in  \cite{zeng2016sparse}, with a thorough proof in \cite[Theorem 7]{wang2022extrapolated}.  It claims that a smooth function has a KL exponent $\theta =1/2$ at a non-degenerate critical point, where the Hessian is non-singular. Based on this, we can derive the following result.

\begin{proposition}\label{calculate.KL}  Consider the following four cases.  
\begin{enumerate}
\item[(i)]  The KL exponent of $F(x,\varepsilon)$ restricted on $\Mcal(x^*,\varepsilon^*)$ at $(x_{\Ical^*}^*, 0)$ is $\theta $. 
\item[(ii)]  The KL exponent of $F(x,\varepsilon)$ at $(x^*, 0)$ is $\theta $. 
\item[(iii)]  The KL exponent of $F(x)$ restricted on $\overline \Mcal(x^*)$ at $x_{\Ical^*}^*$ is $\theta $. 
\item[(iv)]   The KL exponent of $F(x)$ at $x^*$ is $\theta $. 
 \end{enumerate} 
 Then we have (i) $\Longleftrightarrow$ (ii), (iii) $\Longleftrightarrow$ (iv), and (i)$\implies$(iii).  Moreover, we 
 have $\theta  \in (0,1)$ and  $\theta = 1/2$   if $\nabla^2_{\Ical^*\Ical^*} F(x^*)$ is nonsingular in (iii).   
\end{proposition}

\begin{proof} 
Since $F(x,\varepsilon)$ is restricted on $\Mcal(x^*,\varepsilon^*)$, there exists a sufficiently small $\rho>0$ such that $F(x,\varepsilon)$ is differentiable with $\nabla F$ strictly continuous on $\mathcal{B}((x_{\Ical^*}^*,0),\rho)$. By \cite[Example 3.1]{lewis2002active} and \cite[Proposition 13.33]{rockafellar2009variational}, we know that $F$ is prox-regular \cite[Definition 13.27]{rockafellar2009variational} and $\Ccal^2$-partly smooth \cite[Definition 2.7]{lewis2002active} at  $(x_{\Ical^*}^{*},0)$. In addition, we have from Assumption \ref{ass.inter} that $0 \in \textrm{rint }\partial F(x_{\Ical^*}^{*},0)$. Therefore, it follows from \cite[Theorem 3.7]{li2018calculus} that (i) $\implies$ (ii). 

Conversely, since $F(x,\varepsilon)$ has the KL exponent $\theta $ at $(x^{*},0)$, there exist $\upsilon>0$, $\rho > 0$ and $c >0$ such that for all $(x^{*},0) \in \textrm{Lev}_{\upsilon}(x^{*},0)$ with $\textrm{Lev}_{\upsilon}(x^{*},0) = \{ (x,\varepsilon) \in \mathcal{B}((x^{*},0),\rho) \mid F(x^{*},0)<F(x,\varepsilon)<F(x^{*},0)+\upsilon\}$,
    $\textrm{dist}(0,\bar{\partial} F(x,\varepsilon))\geq c (F(x,\varepsilon) - F(x^{*},0))^{\theta }.$
Since $x_i^{*} \neq 0$ for any $i \in \Ical^*$, there exists $\hat{\rho} >0$ such that for all $(z,0) \in \mathcal{B}(x_{\Ical^*}^{*},0)$, $z_i \neq 0, i \in \Ical^*$. Let $\tilde{\rho} = \min(\rho, \hat{\rho})$. Consider any $(x,\varepsilon) \in \Mcal(x^*,\varepsilon^*)$. It implies that
$(x_{\Ical^{*}},\varepsilon_{\Ical^{*}}) \in \textrm{Lev}_{\upsilon}(x^{*}_{\Ical^{*}},0)$ with $\textrm{Lev}_{\upsilon}(x^{*}_{\Ical^{*}},0) = \{(x_{\Ical^*},\varepsilon_{\Ical^*})\in\mathcal{B}((x_{\Ical^*}^{*},0),\tilde{\rho})\mid F(x_{\Ical^*}^{*},0) \leq F(x_{\Ical^*},\varepsilon_{\Ical^*}) \leq F(x_{\Ical^*}^{*},0) + \upsilon\}$. Hence, we know that
\begin{equation*}
   \begin{aligned}
        \Vert \nabla_{\Ical^{*}}F(x^{*},0)\Vert=\textrm{dist}(0,\bar{\partial} F(x,\varepsilon)) &\geq c (F(x,\varepsilon) - F(x^{*},0))^{\theta }\\
        &= c (F(x_{\Ical^{*}},\varepsilon_{\Ical^{*}}) - F(x_{\Ical^*}^{*},0))^{\theta }.
   \end{aligned}
\end{equation*}
This completes the proof of conclusion (i) $\Longleftrightarrow$ (ii). The proof of (iii) $\Longleftrightarrow$ (iv) can be finished using similar arguments by setting $\varepsilon \equiv 0$.

We now prove (i) $\implies$ (iii). Note that $\varepsilon_{\Ical^*}^* = 0$ and $\frac{\partial}{\partial \varepsilon_{\Ical^*}}  F(x^*, \varepsilon^*) = 2\omega_{\Ical^*}^*\circ \varepsilon_{\Ical^*}^* = 0$.  By KL property and (i),  there exists $c > 0$, such that 
\begin{equation*}
    \begin{aligned}
        \| \nabla_{{\Ical^*}  } F(x)\| =&\  \| 
        \left[
        \begin{matrix}
            \frac{\partial}{\partial x_{\Ical^*}} F(x, \varepsilon^*) \\
            0
        \end{matrix}
        \right] \| =  \|\nabla_{{\Ical^*}  } F(x,\varepsilon^*)\| \\
        \geq &\ c(F(x, \varepsilon^{*}) - F(x^*,\varepsilon^*))^{\theta } \ge c(F(x) - F(x^*))^{\theta }
    \end{aligned}
\end{equation*}
indicating (iii) holds.

Moreover, if $\nabla^2_{\Ical^*\Ical^*} F(x^*)$ is nonsingular,   \cite[Theorem 6]{wang2022extrapolated} indicates that the KL exponent of (c) is $ 1/2$. In addition, 
$\| \nabla_{\Ical^*} F(x)\| \ge  c(F(x) - F(x^*))^0$ cannot be true by Corollary \ref{coro.grad}, so that $\theta \ne 0$. 

\end{proof}

The convergence rate analysis of IR$\ell_1$-type methods for $\ell_p$-regularized problems \eqref{p.prob} under the KL property has been studied in \cite{wang2023convergence, wang2022extrapolated}. For further insights into convergence rate analysis for a broader range of descent methods, we refer readers to \cite{li2018calculus, attouch2013convergence, attouch2009convergence}.
We proceed to demonstrate that IReNA satisfies both the \textit{sufficient decrease condition} and the \textit{relative error condition} as outlined in \cite{attouch2013convergence, attouch2009convergence, wen2018proximal}.
\begin{lemma}
Let $\{(x^k,\epsilon^k)\}$ be the sequence generated by Algorithm \ref{algo.IReNA}. Then the following statements hold.
 \begin{itemize}
     \item[(i)] (sufficient decrease condition) There exists $C_1 > 0$ such that
     \begin{equation*}\label{eq.F_x}
        F(x^{k+1}, \varepsilon^{k+1}) +  C_1 \|x^{k+1} - x^{k}\|^2  \le F(x^{k }, \varepsilon^{k }).
    \end{equation*}
 
    \item[(ii)] (relative error condition) There exists $C_2>0$ such that
    \begin{equation*}\label{eq.rec}
    \|\nabla F(x^{k+1},\varepsilon^{k+1})\| \leq C_2 (\|x^{k+1} - x^k\| + \|\varepsilon^{k}\|_1 - \|\varepsilon^{k+1}\|_1).
\end{equation*}
 \end{itemize}
\end{lemma}
\begin{proof} 
 For brevity and without loss of generality, we remove the subscript $\Ical^*$ in this proof. Write $\nabla_{x} F(x,\varepsilon)$ for $\frac{\partial}{\partial x_{\Ical^*}} F(x,\varepsilon)$ and $\nabla_{\varepsilon} F(x,\varepsilon)$ for $\frac{\partial}{\partial \varepsilon_{\Ical^*}} F(x,\varepsilon)$.  

(i). It follows from \eqref{pls.des} that $C_1 = \frac{\xi_{\alpha}}{2}$.

(ii). We first establish an upper bound of $\|\nabla_{x} F(x^{k+1},\varepsilon^{k+1}) \|$.
\begin{equation*}
    \begin{aligned}
        &\|\nabla_{x} F(x^{k+1},\varepsilon^{k+1}) \| \leq \|\nabla_{x} F(x^{k+1},\varepsilon^{k+1}) -  \nabla_{x} F(x^{k},\varepsilon^{k})\| + \|\nabla_{x} F(x^{k},\varepsilon^{k})\| \\
        \leq &\ \|\nabla f(x^{k+1})-\nabla f(x^k)\| + \|\omega(x^{k+1},(\varepsilon^{k+1})^2) - \omega(x^{k},(\varepsilon^{k})^2)\| + \|\nabla_{x} F(x^{k},\varepsilon^{k})\|. \\
    \end{aligned}
\end{equation*}
By the Lipschitz property of $f$, the first term is bounded by
\begin{equation*} 
    \|\nabla f(x^{k+1})-\nabla f(x^k)\| \leq L_f \|x^{k+1}-x^k\|.
\end{equation*}
Now we give an upper bound for the third term. Combining \eqref{eq.xij2.3} and \eqref{alpha is bound} we have for $k \ge \check{k}$, 
   \begin{equation*}
   \begin{aligned}
        \|x^{k+1} - x^k\| =\| \alpha_k \tilde{d}^k \| &\ge \|\underline{\alpha} \tilde{d}^k\|  \ge     \frac{\underline{\alpha}}{\lambda_{\max}}     \|\nabla_{x} F(x^k,\varepsilon^k) \|.
   \end{aligned}
   \end{equation*} 
Finally, we give an upper bound for the second term. From \cite[Lemma 1]{wang2023convergence}, we have the magnitude change between $\omega^{k+1} := \omega(x^{k+1},(\varepsilon^{k+1})^2)$ and $\omega^{k} := \omega(x^{k},(\varepsilon^{k})^2)$ is bounded by:
\begin{equation*}
    \|\omega^{k+1} - \omega^k\| \leq \bar{C} [\sqrt{|\Ical^*|}\|x^{k+1} - x^{k}\| + 2\|\varepsilon^0\|_{\infty}(\|\varepsilon^{k}\|_1 - \|\varepsilon^{k+1}\|_1)],
\end{equation*}
where {    $\bar{C} $ is the local Lipschitz constant of $r'$.}
Putting together the bounds for all three terms, we have
\begin{equation}\label{eq.xgrad}
\begin{aligned}
        \|\nabla_{x} F(x^{k+1},\varepsilon^{k+1}) \| \leq& (L_f + \frac{\lambda_{\max}}{\underline{\alpha}} +  \bar C \sqrt{|\Ical^*|} )\|x^{k+1} - x^{k}\| \\
    &+  2\|\varepsilon^0\|_{\infty}\bar C (\|\varepsilon^{k}\|_1 - \|\varepsilon^{k+1}\|_1).
\end{aligned}
\end{equation}
On the other hand, 
\begin{equation}\label{eq.epsgrad}
    \begin{aligned}
        &\quad\ \|\nabla_{\varepsilon} F(x^{k+1},\varepsilon^{k+1}) \|\\
        &\leq \|\nabla_{\varepsilon} F(x^{k+1},\varepsilon^{k+1}) \|_1 = \sum_{i \in \Ical^*} 2[\omega(x^{k+1},(\varepsilon^{k+1})^2)]_i \cdot \varepsilon_i^{k+1} \\
        &\leq  \sum_{i \in \Ical^*} 2 \hat{\omega} \frac{\sqrt{\gamma_\epsilon}}{1-\sqrt{\gamma_\epsilon}} (\varepsilon_i^{k}-\varepsilon_i^{k+1}) \leq 2 \hat{\omega} \frac{\sqrt{\gamma_\epsilon}}{1-\sqrt{\gamma_\epsilon}} (\|\varepsilon^{k}\|_1 - \|\varepsilon^{k+1}\|_1),
    \end{aligned}
\end{equation}
where the second inequality holds by Proposition \ref{Prop.support.stable}(i) and $\varepsilon^{k+1} \leq \sqrt{\gamma_\epsilon} \varepsilon^k$. Overall, we obtain from \eqref{eq.xgrad} and \eqref{eq.epsgrad} that 
\begin{equation*}
    \|\nabla F(x^{k+1},\varepsilon^{k+1})\| \leq C_2 (\|x^{k+1} - x^k\| + \|\varepsilon^{k}\|_1 - \|\varepsilon^{k+1}\|_1),
\end{equation*}
where
$$
C_2 := \max\{L_f + \frac{\lambda_{\max}}{\underline{\alpha}} +  \bar C \sqrt{|\Ical^*|},  2\|\varepsilon^0\|_{\infty}\bar C + 2 \hat{\omega} \frac{\sqrt{\gamma_\epsilon}}{1-\sqrt{\gamma_\epsilon}} \}.
$$ The proof is complete.
\end{proof}
With the above results established, the subsequent analysis follows naturally and can be derived using standard arguments.
\begin{theorem}\label{lem.acc.gc2.conv}
Let $\{(x^k,\epsilon^k)\}$ be the sequence generated by Algorithm \ref{algo.IReNA}. Suppose that  $F(x,\varepsilon)$ restricted to $ \Mcal(x^*, \varepsilon^*)$ is a KL function at all stationary points $(x^*,0)$. Then it holds that
\begin{equation}\label{kl_conv}
    \sum_{k=0}^{\infty} \Vert x^{k+1} -x^{k} \Vert < \infty.
\end{equation}
Therefore $\{x^{k}\}$ converges to a stationary point $x^{*}$.
In addition, if $F$ is a local KL function of some exponent $\theta  \in (0,1)$ at all stationary points. Then for sufficiently large $k$ it holds that
	\begin{enumerate}
		\item[(i)] If $\theta  \in (0,\frac{1}{2}]$, then there exists $\vartheta \in (0,1), C_3>0$  such that $\|x^k-x^*\|< C_3\vartheta^k$.
		\item[(ii)] If $\theta  \in (\frac{1}{2},1)$, then there exist $C_4>0$ such that $\|x^k-x^*\|< C_4 k^{-\frac{1-\theta }{2\theta -1}}$.
	\end{enumerate}
\end{theorem}

\begin{proof}
Statements \eqref{kl_conv}, (i), and (ii) follow from the standard analysis of perturbed functions with the KL property. Consequently, \eqref{kl_conv} can be derived from the convergence analysis in \cite[Theorem 2.9]{attouch2013convergence} or \cite[Theorem 4]{attouch2009convergence} and (i),(ii) can be derived using the same arguments as in the proofs of \cite[Theorem 4]{wang2023convergence} and \cite[Theorem 4]{wang2022extrapolated}.
\end{proof}

\subsection{Local Convergence Analysis with Exact QP Solution}
In this subsection, we establish the local convergence properties of Algorithm \ref{algo.IReNA} in the neighborhood of stationary points. By Proposition \ref{Prop.support.stable}, after a finite number of iterations, the iterates generated  by IReNA are obtained through solving only a reduced-space QP subproblem combined with a backtracking line  search.  For the purpose of this analysis, let $\chi^{\infty}$ be the set consisting of all stationary points of $\{x^{k}\}$ generated by IReNA, and we make the following assumptions.

\begin{assumption}\label{assumption.2} 
Let $\{x^k\}$ be the sequence generated by Algorithm \ref{algo.IReNA} with $\{x^k\} \to x^{*} \in \chi^{\infty}$. Suppose there exists $M >0$ such that reduced-space Hessian of $F(x)$ at $x^{*}$ satisfying $\| \nabla^2_{\Ical^*\Ical^*}F (x^{*})^{-1}\| \leq M$. For all sufficiently large $k$, we assume
\begin{itemize} 
         \item[(i) ] The exact Hessian  $H^k =\nabla^2_{\Ical^k\Ical^k}F (x^{k};\epsilon^k)$  is used in $m_k(d)$ and $\Wcal_k \equiv \Ical^k = \Ical(x^{*})$. 
        \item[(ii)] The reduced-space QP subproblem \eqref{sub.qp} is solved exactly i.e., $\tilde{d}^k =  (H^k)^{-1} g^k$.
         \item[(iii)] For all sufficiently large $k$, the unit stepsize $\alpha_k \equiv 1$ is accepted.
    \end{itemize}
    \end{assumption} 
    As shown in Proposition \ref{calculate.KL} ,  the non-singularity of $\nabla^2_{\Ical^*\Ical^*} F(x) $ implies that the KL exponent of $F$ is $1/2$ at $x^*$. Therefore, a linear convergence rate is achieved by Theorem \ref{lem.acc.gc2.conv}.  In the following, we show that under Assumption \ref{assumption.2}, a quadratic convergence rate can be attained when $\epsilon_{\Ical^*}^k \to 0$ at a quadratic rate.

\begin{theorem} \label{theorem.quadratic} 
    Suppose Assumption \ref{assumption.2} holds. Let $\{(x^{k},\epsilon^{k})\}$ be the sequence generated by Algorithm \ref{algo.IReNA} and $\{x^{k}\} \to x^{*} \in \chi^{\infty}$. Then there exit $\check{k} \in \mathbb{N}$ and $\rho>0$ such that for all $x^{k} \in \mathcal{B}(x^{*},\rho)$ satisfies
    \begin{equation*}
        \|x^{k+1} - x^{*}\| \leq  \frac{3M L_H}{2} \|x^k-x^{*}\|^2 + \mathcal{O}(\|\epsilon^{k}\|),\ \  \forall  k \geq \check{k},
    \end{equation*}
    where $\nabla^2_{\Ical^*\Ical^*} F(x)$ is locally Lipschitz continuous with constant $L_H>0$ on $\mathcal{B}(x^{*},\rho)$.
    \end{theorem}
\begin{proof}     	 
We can remove the subscript $\Wcal_k$ for simplicity by Assumption \ref{assumption.2}(i). Given $\nabla^2 F(x^*)$ is nonsingular, we can select a sufficiently small $\rho$ so that $\nabla^2 F(x^k; \epsilon^k)$ is also nonsingular for any $\|x^k-x^*\|<\rho$ and $\|\epsilon^k\|<\rho$, since $\nabla^2 F(\cdot; \epsilon)$  is continuous with respect to $\epsilon$. We have from Assumption \ref{assumption.2}(ii) that
		\begin{equation*}\label{eq.qdeq}
			\begin{aligned}
				x^{k+1} - x^{*} &= x^k - x^{*} - \nabla^2 F(x^k; \epsilon^k)^{-1}\nabla F(x^k; \epsilon^k)\\
				&=  x^k - x^{*} - \nabla^2 F(x^k; \epsilon^k)^{-1}(\nabla F(x^k; \epsilon^k) - \nabla F(x^{*})) \\
				&= \nabla^2 F(x^k; \epsilon^k)^{-1}(\nabla F(x^{*}) - \nabla F(x^k; \epsilon^k) - \nabla^2 F(x^k; \epsilon^k)(x^{*} - x^k)).
			\end{aligned}
		\end{equation*}
It follows that
		\begin{equation}\label{local.big.ineq}
			\begin{aligned}
				&\| x^{k+1} - x^{*}\|  \\
                \leq & \  \| \nabla^2 F(x^k; \epsilon^k)^{-1}\|  \| \nabla F(x^{*}) - \nabla F(x^k; \epsilon^k) - \nabla^2 F(x^k; \epsilon^k)(x^{*} - x^k)\|  \\
                = &\ \| \nabla^2 F(x^k; \epsilon^k)^{-1}\|    \| \nabla F(x^{*}) - \nabla F(x^k) - \nabla^2 F(x^k)(x^{*} - x^k) \\
                &+\nabla F(x^k) - \nabla F(x^k; \epsilon^k) + [\nabla^2 F(x^k)- \nabla^2 F(x^k; \epsilon^k)](x^{*} - x^k) \| .
			\end{aligned}
		\end{equation}		
Consider any $\rho > 0$ and $x^k$ such that  $\| x^k - x^{*}\| \leq \rho < \frac{1}{2ML_H}$. We have that  
		\begin{equation*}
			\begin{aligned}
				\| \nabla^2 F(x^{*})^{-1}(\nabla^2 F(x^k) - \nabla^2 F(x^{*}))\|  &\leq \| \nabla^2 F(x^{*})^{-1}\|    \| \nabla^2 F(x^k) - \nabla^2 F(x^{*})\|  \\
				&\leq ML_H \| x^k - x^{*}\|  \leq \rho ML_H  \leq \tfrac{1}{2}, 
			\end{aligned}
		\end{equation*}
        and 
		\begin{equation*}
			\| \nabla^2 F(x^k)^{-1}\|  \leq \frac{\| \nabla^2 F(x^{*})^{-1}\| }{1-\| \nabla^2 F(x^{*})^{-1}(\nabla^2 F(x^k) - \nabla^2 F(x^{*}))\| } \leq 2M.      \footnote{For matrices $a$ and $B$, if $a$ is nonsingular and $\|a^{-1}(B-a)\|_2<1$, then $B$ is nonsingular and 
        $\|B^{-1}\|_2\leq \frac{\|a^{-1}\|_2}{1-\|a^{-1}(B-a)\|_2}.$}
		\end{equation*}
	We can then choose $\rho$ even smaller such that   
	\begin{equation}\label{local.bd.1}
	\| \nabla^2 F(x^k; \epsilon^k)^{-1}\| <  3M,
	\end{equation}
	 for any $\|x^k-x^*\|<\rho$ and $\|\epsilon^k\|<\rho$.

{   On the other hand, for $x^k \in \Bcal(x^*,\rho)$ and $|x^k| > \delta > 0$, there exist two constant $c_1,c_2$, such that $r'$ is locally $c_1$-Lipschitz continuous and $r''$ is locally $c_2$-Lipschitz continuous. It follows that
      \begin{equation*}
          \begin{aligned}
       |  \lambda r'(|x^k_i|) - \lambda r'(|x^k_i|+\epsilon^k_i) | 
        &\le  c_1 \epsilon^k_i, \\
        |  \lambda r''(|x^k_i|) - \lambda r''(|x^k_i|+\epsilon^k_i) |  &\le  c_2 \epsilon^k_i.
        \end{aligned}
      \end{equation*} 
      Therefore, it holds that   
        \begin{equation}\label{local.bd.2}
        \begin{aligned}
            \|\nabla F(x^k) - \nabla F(x^k; \epsilon^k)\| &\leq 
        c_1 \|\epsilon^k\|, \\
                \| [\nabla^2 F(x^k)- \nabla^2 F(x^k; \epsilon^k)](x^{*} - x^k)\| 
        &\leq c_2\| \epsilon^k\|\cdot \|x^k \!-\! x^{*}\|.
        \end{aligned}
        \end{equation}}
Since $\nabla F(x)$ is continuously differentiable on $\mathcal{B}(x^{*},\rho)$, we know   
        \begin{equation}\label{local.bd.4}
        \|\nabla F(x^{*}) - \nabla F(x^k) - \nabla^2 F(x^k)(x^{*} - x^k)\| \leq \frac{L_H}{2} \|x^k-x^{*}\|^2.
        \end{equation}
Combining \eqref{local.big.ineq}  with  \eqref{local.bd.1},  \eqref{local.bd.2} and \eqref{local.bd.4}, we have for all $\|x^k-x^*\|<\rho$ and $\|\epsilon^k\|<\rho$ that 
        \begin{equation*}
            \begin{aligned}
                 \| x^{k+1} - x^{*}\| \leq &\ 3M ( \frac{L_H}{2} \|x^k-x^{*}\|^2 + c_1 \|\epsilon^k\| +  c_2\| \epsilon^k\|   \|x^k - x^{*}\| )\\
                \leq &\  \frac{3M L_H}{2} \|x^k-x^{*}\|^2 + \mathcal{O}(\|\epsilon^{k}\|).
            \end{aligned}
        \end{equation*}
        This completes the proof.
\end{proof}

\subsection{Local Second-order Complexity Grantees with Trust-region Subproblem}\label{sec.tr} 
The QP subproblem solved in line \ref{line.QP} of Algorithm \ref{algo.IReNA} seeks an inexact Newton direction restricted to a reduced space using subspace regularized Newton method. Proposition \ref{Prop.support.stable} indicate that the QP subproblem will be always triggered after some iterations. Therefore, we can consider replacing \eqref{sub.qp} with the following tailored regularized trust-region subproblem:
\begin{equation}\label{eq.trustregion}
    \begin{aligned}
        \tilde{d}_{\Wcal_k}^k = \underset{d\in\mathbb{R}^{|\Wcal_k|}}{\arg\min}\quad  \langle g^k, d\rangle+ \frac{1}{2}\langle d, H^k d\rangle + \frac{1}{2} \zeta_H \|d\|^2, \ \ \ \ \text{s.t.  }\|d\| \leq \rho_{k},
    \end{aligned}
\end{equation}
where $\zeta_H > 0$ and $\rho_{k} > 0$ represents  to the trust-region radius and $H^{k}$ may be indefinite. The trust-region Newton method presented in \cite[Algorithm 2.1]{curtis2021trust} can be directly applied to solve \eqref{eq.trustregion} with a slight modification. Specifically, we solve the regularized trust-region subproblem \eqref{eq.trustregion} and set $\tilde{d}_{[\Wcal_k]^c}^k = 0$ to obtain an intermediate direction $\tilde{d}^{k}$. We then compute the trial step $d^{k} = \mathbb{P}(x^{k}+\tilde{d}^{k};x^{k}) - x^{k}$. If $d^{k}$ leads to a non-increasing objective, it is accepted; otherwise, the radius of the trust region is reduced. In addition, it can be straightforwardly verified that Theorem \ref{theo.global.conv} also holds. Therefore, a  similar local second-order complexity can be derived, which is stated as follows.
\begin{theorem}\cite[Theorem 2.6]{curtis2021trust}
    Suppose the termination condition line  \ref{line.terminate} in Algorithm \ref{algo.IReNA} is set to $ \|\nabla F(x^k;\epsilon^k)\| \leq \tau \textrm{  and  } \lambda_{\min}(\nabla^2 F(x^k;\epsilon^k)) \geq -\tau^{1/2}.$ Let $\check{k} \in \mathbb{N}$ such that $\Ical(x^k) = \Ical^*$ for all $k > \check{k}$. Then the trust-region subproblem is implemented, and it holds that $|\Kcal| = \tilde{\Ocal}(\tau^{-3/2})$ with $\Kcal := \{k|\ k> \check{k}\}$.
\end{theorem}
{   The notation $s(\cdot) = \tilde{\mathcal{O}}(t(\cdot))$ means $\vert s(\cdot) \vert \leq C t(\cdot)\vert \log^{c}(\cdot)\vert$ for some $C \in (0,+\infty)$ and $c \in (0,+\infty)$.}

\section{Numerical Experiments} \label{sec.experiments}
In this section, we deliver a set of numerical experiments using both synthetic and real-world data across diverse model prediction problems to demonstrate the effectiveness of the proposed IReNA. We begin by considering the nonconvex $\ell_p$-regularized logistic regression problem \cite{liu2007sparse}: 
\begin{equation}\label{Prob.logistic regression}%
    \begin{aligned}
		& \underset{x \in \mathbb{R}^n}{\textrm{min}}
		& &  \sum_{i=1}^m \log(1+e^{-a_i x^T b_i})+ \lambda\|x\|_p^p,
    \end{aligned}
\end{equation}
where $\lambda >0$, and $a_i \in \{-1,1\}$ are the labels, and $b_i \in \mathbb{R}^n$ are the feature vectors for $i \in [m]$. We also incorporate various nonconvex regularizers into the logistic regression problem to further demonstrate the efficiency of the proposed algorithm.

In these numerical studies, we compare the proposed IReNA with several state-of-the-art algorithms for solving \eqref{Prob.logistic regression}, including HpgSRN \cite{wu2023regularized}\footnote{\href{https://github.com/YuqiaWU/HpgSRN}{The source code is available at https://github.com/YuqiaWU/HpgSRN}}, PCSNP\cite{zhou2023revisiting}\footnote{\href{https://github.com/YuqiaWU/HpgSRN}{The source code is available at https://github.com/ShenglongZhou/PSNP}} and EPIR$\ell_1$ \cite{wang2022extrapolated}. All implementations were carries out in MATLAB and executed on a PC with an Intel i$9$-$13900$K $3.00$ GHz CPU and $64$ GB of RAM.

\subsection{Experimental Setup}\label{param}
In our experiments, we set $\lambda = 1$ and obtain labels $a$ and features $b$ from different datasets. The synthetic dataset is generated following the methods in \cite{chen2017reduced,keskar2016second}. Specifically, for $ i \in [m]$, labels $a_i$ are drawn from a {B}ernoulli distribution and feature vectors $b_i$ are sampled from a standard Gaussian distribution. The real-world datasets--- \textit{w8a}, \textit{a9a}, \textit{real-sim}, \textit{gisette}, \textit{news20}, and \textit{rcv1.train}---are binary classification examples obtained from the LIBSVM repository\footnote{\href{https://www.csie.ntu.edu.tw/~cjlin/libsvmtools/datasets/}{Refer to https://www.csie.ntu.edu.tw/~cjlin/libsvmtools/datasets/}}. 

For IReNA, we set the following parameters:
\begin{equation*}
    \eta= 1 ,\eta_{\Phi}= 1,\eta_{\Psi} = 1, 
 \gamma_{\alpha} = 0.5,\gamma_{\beta} = 0.5,   \xi_{\beta}  = 10^{-8},\xi_{\alpha}  = 10^{-8}, \tau  = 10^{-6}, \nu = 1.
\end{equation*}
For the subspace quadratic subproblem \eqref{sub.qp}, we adopt the truncated conjugate gradient method described in \cite[Section 4.2]{chen2017reduced}. In this subproblem, we set 
\begin{equation*}
    \zeta_k = 10^{-8}+10^{-4} \|g^k\|^{0.5} + \min \{-\lambda r''(x^k_i)\ |\ i \in \Ical(x^k)\},
\end{equation*}
to ensure $H^k \succ 0$. For the tailored trust-region Newton subproblem discussed in Section \ref{sec.tr}, we employ the solver from \cite[Algorithm 2]{curtis2021trust}. In line \ref{line.IST} of Algorithm \ref{algo.IReNA}, the initial stepsize $\bar{\beta}$ is determined using  the classic Barzilai-Borwein (BB) rule \cite{barzilai1988two}.

We initialize the perturbation values as $\epsilon_i^0 = 1$ for all $i \in [n]$ and update the perturbation $\epsilon$ in line \ref{line.reweight} according to:
\begin{equation*}
    \epsilon_i^{k+1} = \begin{cases}
        0.9 \epsilon_i^{k} , & k \in \Scal_{\textrm{IST}} \textrm{ and } \Wcal^k \textrm{ is chosen from \eqref{eq.chooseW1}} , i \in \Ical^{k+1},\\
        0.9 (\epsilon_i^{k})^{1.1}, & k \in  \Scal_{\textrm{IST}} \textrm{ and } \Wcal^k \textrm{ is chosen from \eqref{eq.chooseW2}}  , i \in \Ical^{k+1},\\
        \min\{0.9\epsilon_i^{k} ,(\epsilon_i^{k})^2\}, &k \in \Scal_{\textrm{QP}}, i \in \Ical^{k+1},\\
           \epsilon_i^{k} , &\textrm{otherwise.}
    \end{cases}
\end{equation*}
Furthermore, before iteration $k$ enters $\Scal_{QP}$, we set $\epsilon^k_i = \max\{\epsilon^k_i, 10^{-8}\}$ for all $i$ when $k \in \Scal_{\textrm{IST}}$. This precaution ensures that the perturbation values do not become excessively small.

The parameters for HpgSRN, PCSNP, and EPIR$\ell_1$, are chosen according to the suggestions provided in their respective papers. In particular, because PCSNP addresses an $\ell_2$-norm regularized logistic regression problem, we set the $\ell_2$-norm regularization parameter to $0$, as specified in that study, to ensure a fair comparison. For all algorithms, the initial point was set as $x_i^{0} = 0$ for each $i\in [n]$. 

\subsection{Numerical Results}
\subsubsection{Local Quadratic Convergence}
We first demonstrate that IReNA exhibits local quadratic convergence using both the synthetic dataset and real-world datasets. Inspired by \cite{burke2014sequential}, our demonstration is based on the following two metrics: 
\begin{equation*}
    \Rcal_{\text{opt}}(x) = \| x \nabla f(x)+\lambda p|x|^p\| _{\infty}, \quad \textrm{ and } \quad\ \Rcal_{\text{dist}}(x) = \| x-x^*\|_{\infty},
\end{equation*}
where $x^{*}$ is the solution returned by IReNA and $p=0.5$. Figure \ref{fig.qudratic} displays $\log_{10}(\Rcal_{\text{opt}})$ and $\log_{10}(\Rcal_{\text{dist}})$ over the last ten iterations of Algorithm \ref{algo.IReNA} for six real-world datasets (see panels (a) and (b)) and synthetic datasets of various sizes (see panels (c) and (d)). For the synthetic datasets, each case was tested over $20$ random trials, and the average results are shown. As illustrated in Figure \ref{fig.qudratic}, the slope  of the convergence curves in the final iterations is generally less than $-2$,   which indicates local quadratic convergence.

\begin{figure}[htbp]
\centering
\begin{subfigure}{0.245\textwidth}
  \centering
  \includegraphics[width=\linewidth]{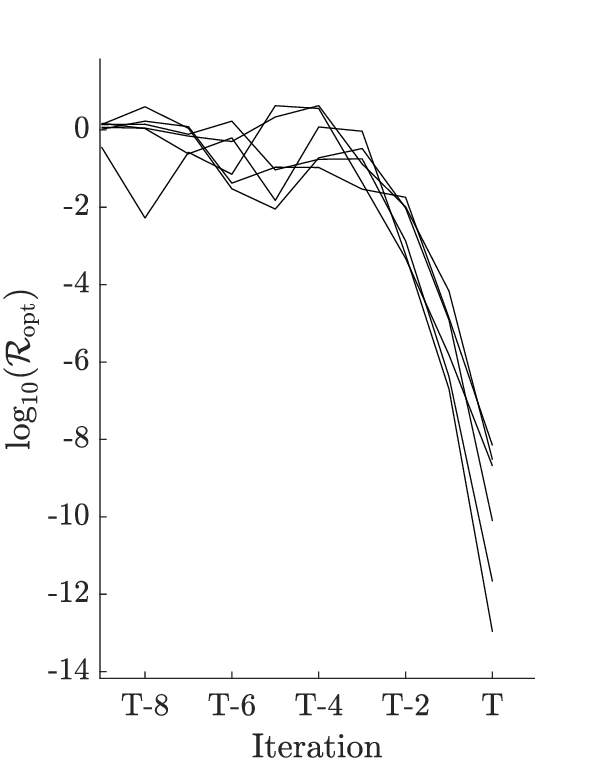}
  \caption{}
\end{subfigure}%
\begin{subfigure}{0.245\textwidth}
  \centering
  \includegraphics[width=\linewidth]{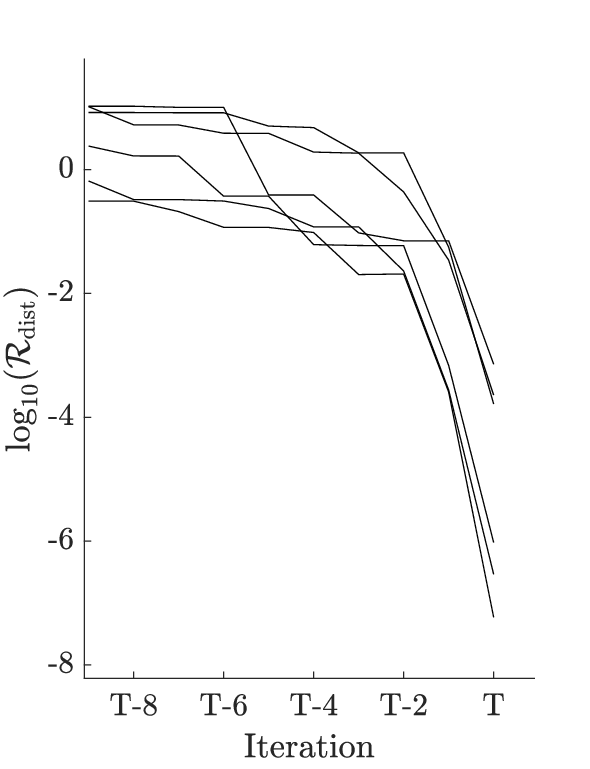}
  \caption{}
\end{subfigure}%
\begin{subfigure}{0.245\textwidth}
  \centering
  \includegraphics[width=\linewidth]{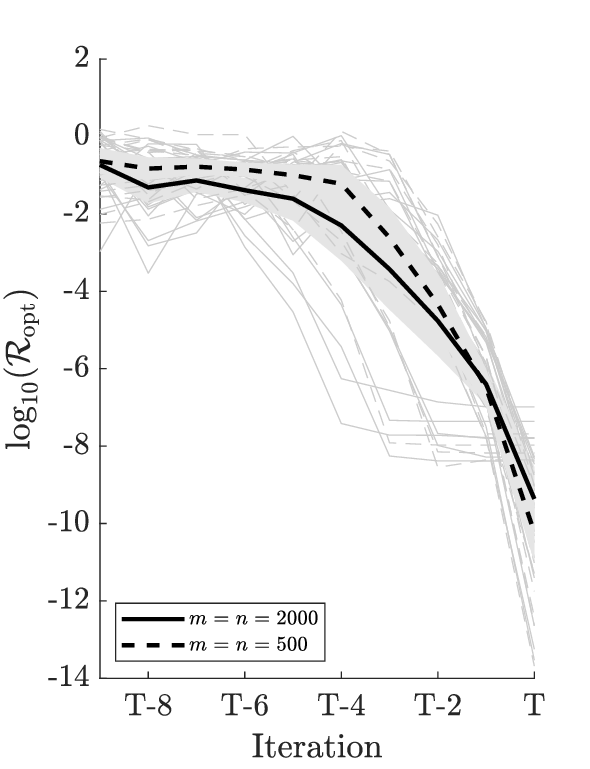}
  \caption{}
\end{subfigure}
\begin{subfigure}{0.245\textwidth}
  \centering
  \includegraphics[width=\linewidth]{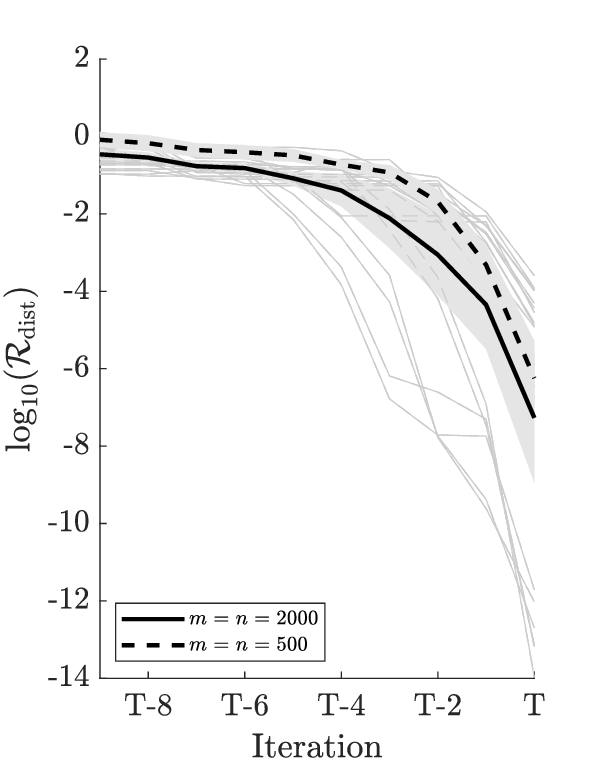}
  \caption{}
\end{subfigure}
\caption{Illustration of the local quadratic convergence of IReNA.}
\label{fig.qudratic}
\end{figure}

\subsubsection{Real-world Datasets}
In this subsection, we compare the performance of IReNA with HpgSRN, PCSNP and EPIR$\ell_1$ on real-world datasets. Note that IReNA-RQP refers to the variant of IReNA that solves a regularized quadratic subproblem \eqref{sub.qp}, whereas IReNA-TR refers to the variant that solves a trust-region subproblem \eqref{eq.trustregion}. For the second-order methods HpgSRN, PCSNP and IReNA, we adopt a termination criterion of
\begin{equation*}
    \| x^k-x^{k-1}\| / \| x^k\| <10^{-9},
\end{equation*}
in addition to their original termination conditions. For the first-order method EPIR$\ell_1$, we maintain its original termination condition of $\| x^k-x^{k-1}\| /\| x^k\| <10^{-4}$.

We evaluate the performance of these methods in terms of CPU time, objective value, and the sparsity level of the obtained solution (e.g., the percentage of zeros). All presented results are averaged over $10$ independent runs. In this test, we set $p = 0.5$. 
 
The comparison results are presented in Table \ref{tab.lr}. As shown in Table \ref{tab.lr}, IReNA-RQP and IReNA-TR generally achieve faster running times while attaining the lowest or near-lowest objective values, and they maintain a sparsity level comparable to that of the benchmark algorithms. During our numerical experiments, we observed that Newton steps in PCSNP are triggered frequently in the early stages, which likely contributes to its longer CPU time.

\begin{table}[htbp]
\small
  \centering
  \caption{Performance comparison on real-world datasets ($p = 0.5$). The feature matrix size is shown in parentheses. The two best objective values and CPU times are highlighted in bold.}
    \begin{tabular}{clcccc}
    \toprule
    \multicolumn{2}{c}{Dataset} & Algorithm & Time (s) & Obj. Values & \% of zeros \\
    \midrule
    \multicolumn{2}{c}{\multirow{2}[1]{*}{a9a}} & IReNA-RQP & \textbf{0.50 } & \textbf{10579.4 } & 45.53\% \\
    \multicolumn{2}{c}{} & IReNA-TR & \textbf{0.42 } & 10588.5  & 46.34\% \\
    \multicolumn{2}{c}{\multirow{3}[1]{*}{ (32561$\times$123)}} & HpgSRN & 2.90  & 10583.3  & 53.01\% \\
    \multicolumn{2}{c}{} & PCSNP & 4.45  & 10599.8  & 53.66\% \\
    \multicolumn{2}{c}{} & EPIRL1 & 5.99  & \textbf{10570.5 } & 39.02\% \\
    \midrule
    \multicolumn{2}{c}{\multirow{2}[1]{*}{w8a}} & IReNA-RQP & \textbf{0.61 } & 5873.8  & 37.00\% \\
    \multicolumn{2}{c}{} & IReNA-TR & \textbf{1.38 } & 5876.9  & 37.00\% \\
    \multicolumn{2}{c}{\multirow{3}[1]{*}{ (49749$\times$300)}} & HpgSRN & 2.01  & \textbf{5856.5 } & 37.00\% \\
    \multicolumn{2}{c}{} & PCSNP & 26.45  & 5870.9  & 37.67\% \\
    \multicolumn{2}{c}{} & EPIRL1 & 13.05  & \textbf{5865.1 } & 38.00\% \\
    \midrule
    \multicolumn{2}{c}{\multirow{2}[1]{*}{gisette}} & IReNA-RQP & \textbf{35.07 } & \textbf{176.4 } & 97.06\% \\
    \multicolumn{2}{c}{} & IReNA-TR & \textbf{32.71 } & \textbf{176.2 } & 97.06\% \\
    \multicolumn{2}{c}{\multirow{3}[1]{*}{ (6000$\times$5000)}} & HpgSRN & 36.58  & 177.0  & 96.92\% \\
    \multicolumn{2}{c}{} & PCSNP & 288.32  & 182.3  & 97.14\% \\
    \multicolumn{2}{c}{} & EPIRL1 & 326.48  & 178.3  & 97.02\% \\
    \midrule
    \multicolumn{2}{c}{\multirow{2}[1]{*}{real sim}} & IReNA-RQP & \textbf{3.81 } & \textbf{7121.6 } & 93.63\% \\
    \multicolumn{2}{c}{} & IReNA-TR & 10.01  & \textbf{7123.3 } & 93.63\% \\
    \multicolumn{2}{c}{\multirow{3}[1]{*}{ (72309$\times$20958)}} & HpgSRN & \textbf{6.80 } & 7262.3  & 94.47\% \\
    \multicolumn{2}{c}{} & PCSNP & 38.73  & 7445.8  & 95.02\% \\
    \multicolumn{2}{c}{} & EPIRL1 & 33.48  & 7152.7  & 93.78\% \\
    \midrule
    \multicolumn{2}{c}{\multirow{2}[1]{*}{rcv1.train}} & IReNA-RQP & \textbf{1.59 } & \textbf{2554.6 } & 99.10\% \\
    \multicolumn{2}{c}{} & IReNA-TR & 3.95  & \textbf{2558.9 } & 99.08\% \\
    \multicolumn{2}{c}{\multirow{3}[1]{*}{ (20242$\times$47236)}} & HpgSRN & \textbf{1.69 } & 2578.4  & 99.21\% \\
    \multicolumn{2}{c}{} & PCSNP & 6.86  & 2652.9  & 99.26\% \\
    \multicolumn{2}{c}{} & EPIRL1 & 14.10  & 2562.9  & 99.13\% \\
    \midrule
    \multicolumn{2}{c}{\multirow{2}[1]{*}{news20 }} & IReNA-RQP & \textbf{20.08 } & 4034.5  & 99.97\% \\
    \multicolumn{2}{c}{} & IReNA-TR & 93.54  & \textbf{3989.7 } & 99.97\% \\
    \multicolumn{2}{c}{\multirow{3}[1]{*}{(19996$\times$1355191)}} & HpgSRN & \textbf{23.74 } & 4171.3  & 99.97\% \\
    \multicolumn{2}{c}{} & PCSNP & 179.61  & 4682.0  & 99.98\% \\
    \multicolumn{2}{c}{} & EPIRL1 & 207.05  & \textbf{3983.6 } & 99.97\% \\
    \bottomrule
    \end{tabular}%
  \label{tab.lr}%
\end{table}%
\subsection{Generic Nonconvex Regularizers}\label{sec.diffp}
\subsubsection{More Stringent $p$ Values}
In this test, we only consider HpgSRN as the benchmark algorithm given the numerical performance comparisons presented in Table \ref{tab.lr}, and we further compare the numerical performance of IReNA-RQP with HpgSRN for solving \eqref{Prob.logistic regression} under more stringent $p$ values. Specifically, we set $p = 0.3$.

Table \ref{tab.smallp} presents the comparison in terms of CPU time and objective value. The table clearly shows that IReNA-RQP significantly outperforms HpgSRN, especially on large datasets such as \textit{real-sim} and \textit{news20}. For these datasets, the CPU time of HpgSRN increases dramatically compared to the results in Table \ref{tab.lr}, whereas our method maintains scalable CPU time. Furthermore, Figure \ref{fig.p3p5} demonstrates this trend on synthetic datasets, where our method requires nearly the same CPU time for $p = 0.3$ as for $p = 0.5$ even as the feature size increases, in contrast to the substantial increases observed with HpgSRN.


\begin{table}[htbp]
	\centering
	\caption{Comparison with HpgSRN on real-world datasets for $p = 0.3$.}
 \setlength{\tabcolsep}{4pt}
    \begin{tabular}{cccccccc}
    \toprule
          & Dataset & news20 & w8a   & a9a   & real sim & rcv1.train & gisette \\
    \midrule
    \multirow{2}[1]{*}{IReNA-RQP} & Time (s)  & 60.21  & 0.80  & 0.47  & 9.01  & 3.86  & 31.79  \\
          & Obj. Values & 3019.06  & 5825.87  & 10596.93  & 5669.32  & 1933.62  & 165.18  \\
    \multirow{2}[1]{*}{HpgSRN} & Time (s) & 1871.86  & 3.29  & 3.49  & 50.67  & 15.73  & 34.07  \\
          & Obj. Values & 3696.13  & 5802.88  & 10595.47  & 6099.26  & 2078.87  & 186.27  \\
    \bottomrule
    \end{tabular}%
	  \label{tab.smallp}%
  \end{table}%

\begin{figure}[htbp]
    \centering
    \includegraphics[width=\linewidth]{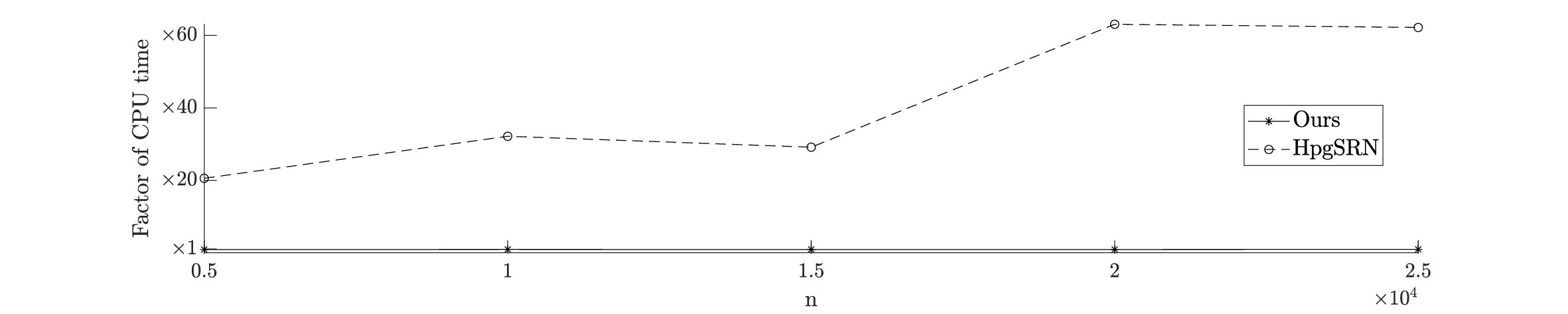}
    \caption{Comparison with HpgSRN for \( p = 0.3 \) on synthetic datasets. The $x$-axis represents the feature size, while the $y$-axis shows the ratio of CPU time for the same problem with \( p = 0.3 \) to that with \( p = 0.5 \). }
    \label{fig.p3p5}
\end{figure}
\subsubsection{Other Nonconvex Regularizers}
To evaluate the performance of IReNA on various nonconvex regularization problems, we incorporate  different nonconvex regularizers, as summarized in Table \ref{tab.l0approx}, into logistic regression using the \textit{a9a}, \textit{gisette}, and \textit{rcv1.train} datasets. In these evaluations, we set $q = 10^{-5}$ for the LOG regularizer and $q = 0.1$ for all others. Figure \ref{fig.extended} presents the logarithmic residuals over the last $10$ iterations, where the residual is defined as
$\Rcal_{\text{opt}} = \| x \circ (\nabla f(x) + \omega(x;0) \circ \text{sgn}(x)) \|_{\infty}$.

\begin{figure}[htbp]
\centering
\begin{subfigure}{0.32\textwidth}
  \centering
  \includegraphics[width=\linewidth]{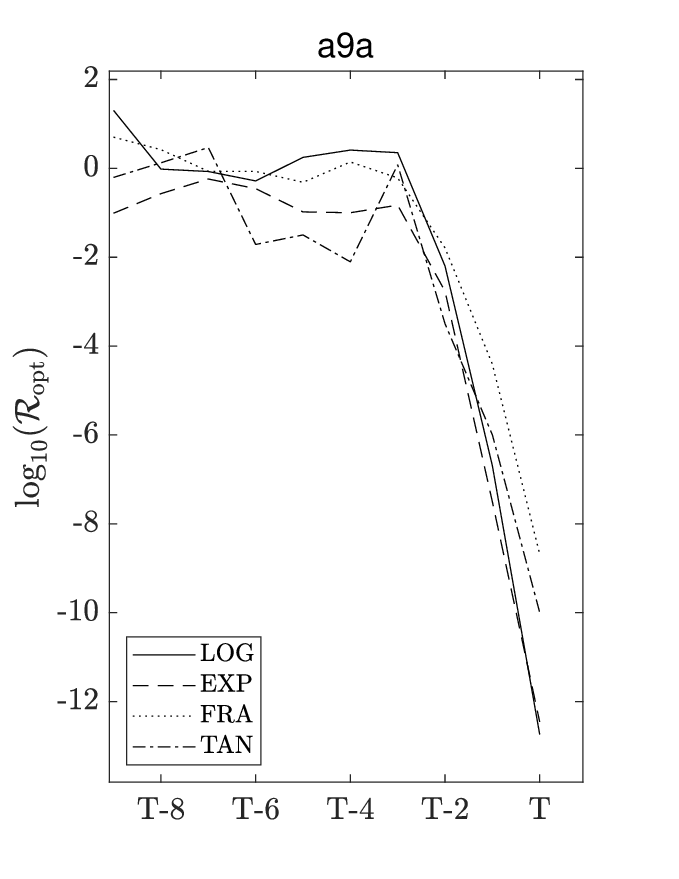}
\end{subfigure}%
\begin{subfigure}{0.32\textwidth}
  \centering
  \includegraphics[width=\linewidth]{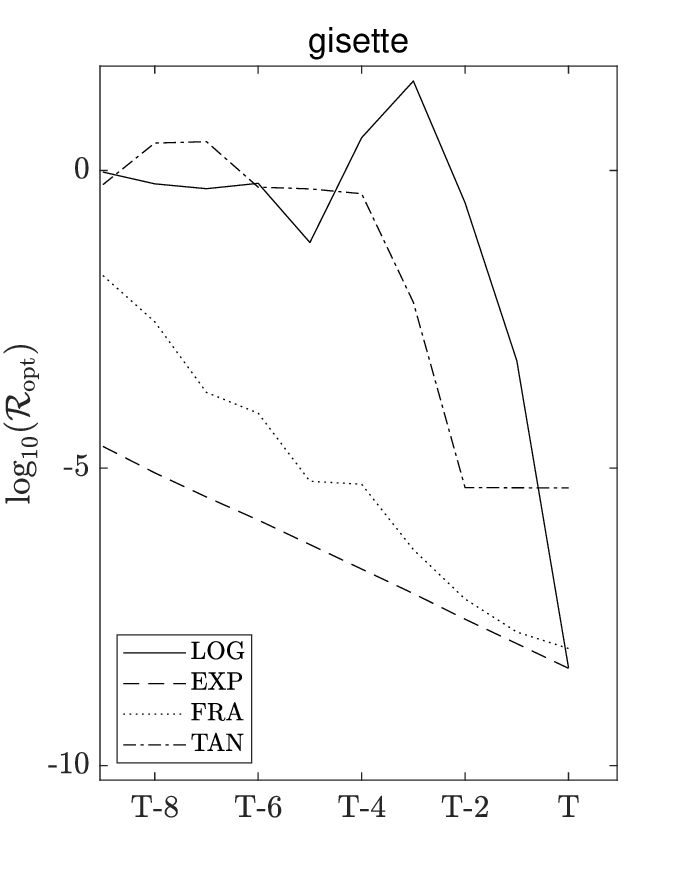}
\end{subfigure}%
\begin{subfigure}{0.32\textwidth}
  \centering
  \includegraphics[width=\linewidth]{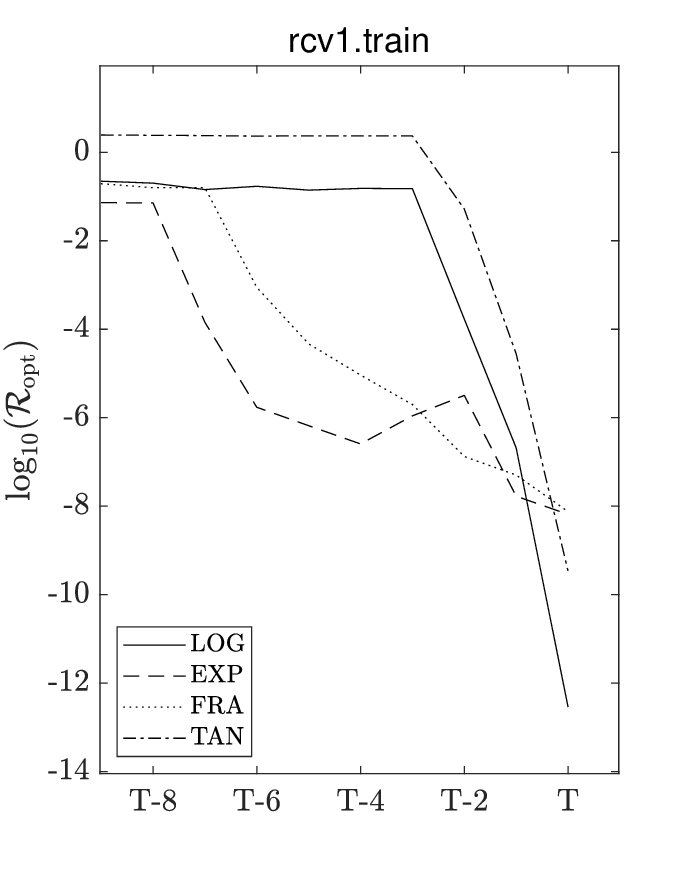}
\end{subfigure}
\caption{Convergence behavior for other nonconvex regularizers on real-world datasets.}
\label{fig.extended}
\end{figure}

\section{Conclusion} 
In this paper, we presented IReNA, a hybrid algorithm designed for a class of nonconvex and nonsmooth sparsity-promoting regularization problems. IReNA adaptively switches between subspace iteratively reweighted $\ell_1$ and subspace regularized Newton iterations based on the sign changes of consecutive iterates. We established the global convergence of the entire sequence of iterates and demonstrated local linear convergence under the KL property. Furthermore, when the subspace quadratic subproblem is solved exactly, the method exhibits local quadratic convergence. Numerical experiments across various  model prediction tasks validated the effectiveness of IReNA, demonstrating its potential for addressing large-scale nonconvex sparse optimization problems.

\section*{Acknowledgments}
The research of Xiangyu Yang was partially supported by National Natural Science Foundation of China No.~12301398 and the China Scholarship Council under Grant No.~202308410343.

\section*{Appendix} 
 \renewcommand{\thelemma}{A}
    \setcounter{equation}{0}

\begin{lemma}\label{lem.ist.ist}
    Let $d(\beta) = \Scal_{ \beta \omega  }(x  -  \beta g ) - x $  for $\beta > 0$ where $x, g\in\mathbb{R}^n$ and $\omega\in\mathbb{R}^n_{++}$. 
    It holds that 
    \begin{subequations}
            \begin{align}
                d_i(\beta)   &=   \beta  d_i(1)  && \textrm{     if  }\  x_i = 0,  \label{le.dk1.ist} \\
                |d_i(\beta)|   & \ge  \min\{\beta ,1\} | d_i(1) | && \textrm{     if  }\ x_i \ne 0. \label{le.dk2.ist}
        \end{align}
    \end{subequations}
              
      Moreover, for $\omega_i > |g_i- x_i/\beta|$, $d_i(\beta) = - x_i$. 
     \end{lemma}
    \begin{proof}
       It follows from the soft-thresholding operator  that
     \begin{equation*}\label{eq.istd.ist}
           d_i(\beta) =
            \begin{cases}
                -\beta(g_i + \omega_i ), & \textrm{if   }\beta(g_i + \omega_i )<x_i,\\
                -\beta(g_i - \omega_i ), & \textrm{if   }\beta(g_i - \omega_i )>x_i,\\
                -x_i, & \textrm{otherwise.}
            \end{cases}
        \end{equation*}
    If $i \in \Ical_0(x )$,  $x_i  = 0$. It is obvious that \eqref{le.dk1.ist} holds. If $i \in \Ical(x )$, $x_i  \ne 0$.   We check the values of $d_i(1)$.  
    
If \( d_i(1) = - (g_i + \omega_i) \), then we have \( g_i + \omega_i < x_i \). Based on the expression of \( d_i(1) \), we analyze the relative order of \( x_i \), \( \beta(g_i + \omega_i) \), and \( \beta(g_i - \omega_i) \). Consider the following three cases:

\begin{itemize}
    \item[(i)] If \( \beta(g_i - \omega_i) < \beta(g_i + \omega_i) < x_i \),  
    this corresponds to the first case in the definition of \( d_i(\beta) \), meaning that  $d_i(\beta) = -\beta(g_i+\omega_i) = -\beta d_i(1)$.
    
    \item[(ii)] If \( \beta(g_i - \omega_i) \leq x_i \leq \beta(g_i + \omega_i) \),  
    this falls under the third case in \( d_i(\beta) \), leading to  $d_i(\beta) = -x_i$. Since \( g_i + \omega_i < x_i \leq \beta(g_i + \omega_i) \), we have two possible subcases:
    \begin{itemize}
        \item If \( 0 < g_i + \omega_i < x_i \), then \( \beta \geq \frac{x_i}{g_i + \omega_i} \geq 1 \).
        \item If \( g_i + \omega_i < x_i < 0 \), then \( \beta \leq \frac{x_i}{g_i + \omega_i} \leq 1 \).
    \end{itemize}
    In either case, it follows that  $|x_i| \geq \min(\beta,1) |g_i + \omega_i|$.
    Consequently, we obtain $|d_i(\beta)| = |x_i| \geq \min(\beta,1) |g_i + \omega_i| = \min(\beta,1) |d_i(1)|.$

    \item[(iii)] If \( x_i < \beta(g_i - \omega_i) < \beta(g_i + \omega_i) \),  
    this corresponds to the second case in \( d_i(\beta) \), giving $d_i(\beta) = -\beta(g_i - \omega_i).$ Since \( g_i + \omega_i < x_i < \beta(g_i - \omega_i) < \beta(g_i + \omega_i) \), we deduce that:
    \begin{itemize}
        \item If \( g_i + \omega_i > 0 \), then \( \beta > 1 \).
        \item If \( g_i + \omega_i < 0 \), then \( \beta < 1 \).
    \end{itemize}
    In either case, we conclude that  $| \beta(g_i - \omega_i) | > \min(\beta,1) | g_i + \omega_i |,$ which implies  $|d_i(\beta)| = | \beta(g_i - \omega_i) | > \min(\beta,1) | g_i + \omega_i |  = \min(\beta,1)  | d_i(1) |.$
\end{itemize}
    
Consider the remaining two cases: $d_i(1) = -( g_i - \omega_i)$ or $d_i(1) =  - x_i$. Using the same argument based on the order of $ x_i, \beta(g_i + \omega_i )$, and $ \beta(g_i - \omega_i )$ , it is straightforward to verify  that \eqref{le.dk2.ist} holds. Furthermore, the latter statement naturally follows for the soft-thresholding operator. This completes the proof.
\end{proof}

\printbibliography
\end{document}